\documentclass[reqno,11pt]{article}
\usepackage{a4wide,xcolor,eucal,enumerate,mathrsfs}
\usepackage[normalem]{ulem}
\usepackage{pdfsync}
\usepackage{hyperref}
\usepackage{amsmath,amssymb,epsfig,amsthm}
\usepackage[latin1]{inputenc}

\numberwithin{equation}{section}

\newtheorem{theorem}{Theorem}[section]

\newtheorem{corollary}[theorem]{Corollary}
\newtheorem{lemma}[theorem]{Lemma}
\newtheorem{proposition}[theorem]{Proposition}

\newtheorem{definition}[theorem]{Definition}

\newtheorem{remark}[theorem]{Remark}

\newcommand{\C}{\mathbb{C}}
\newcommand{\D}{\mathbb{D}}

\renewcommand{\L}{\mathbb{L}}

\newcommand{\R}{\mathbb{R}}

\newcommand{\V}{\mathbb{V}}

\newcommand{\cE}{{\ensuremath{\mathcal E}}}

\newcommand{\cV}{{\ensuremath{\mathbb V}}}

\newcommand{\ff}{{\mbox{\boldmath$f$}}}

\newcommand{\mm}{{\mbox{\boldmath$m$}}}

\newcommand{\ppi}{{\mbox{\boldmath$\pi$}}}

\newcommand{\sfd}{{\sf d}}

\newcommand{\sfg}{{\sf g}}

\newcommand{\sfP}{{\sf P}}

\newcommand{\sfR}{{\sf R}}

\newcommand{\rmc}{{\mathrm c}}

\newcommand{\rme}{{\mathrm e}}

\newcommand{\rmC}{{\mathrm C}}
\newcommand{\rmD}{{\mathrm D}}

\newcommand{\rmI}{{\mathrm I}}

\newcommand{\Kliminf}{K\kern-3pt-\kern-2pt\mathop{\rm lim\,inf}\limits}  
\newcommand{\supp}{\mathop{\rm supp}\nolimits}   

\newcommand{\Lip}{\mathop{\rm Lip}\nolimits}          
\renewcommand{\d}{{\mathrm d}}
\newcommand{\dt}{{\d t}}

\newcommand{\restr}[1]{\lower3pt\hbox{$|_{#1}$}}
\newcommand{\topref}[2]{\stackrel{\eqref{#1}}#2}

\newcommand{\down}{\downarrow}              

\newcommand{\weakto}{\rightharpoonup}

\newcommand{\eps}{\varepsilon}  
\newcommand{\nchi}{{\raise.3ex\hbox{$\chi$}}}

\newcommand{\forevery}{\text{for every }}

\newcommand{\fr}{\hfill$\blacksquare$}                      

\def\qed{\ifmmode 
  \else \leavevmode\unskip\penalty9999 \hbox{}\nobreak\hfill
  \fi               
    \qquad           \hbox{\hskip.5em $\square$
                \hskip.1em}}

\newcommand{\GGG}{\color{blue}}

\renewcommand{\mm}{\mathfrak m} 

\newcommand{\Gbil}[2]{\Gamma\big(#1,#2\big)}   
\newcommand{\Gq}[1]{\Gamma\big(#1\big)}          

\newcommand{\Probabilities}[1]{\mathscr P(#1)}          

\newcommand{\ProbabilitiesTwo}[1]{\mathscr P_2(#1)}     

\newcommand{\fnd}[3]{\sfg_{\ifthenelse{\equal{#2}{}}{}{(#2)}}(#3,#1)}

\newcommand{\Vhom}[1]{\V_{\kern-2pt
\ifthenelse{\equal{#1}{1}}{\cE}{#1}}}
\newcommand{\Vdual}[1]{\V_{\kern-2pt
\ifthenelse{\equal{#1}{1}}{\cE}{\rm #1}}'}
\newcommand{\Ddual}{\D_{\kern-1pt \cE}'}

\newcommand{\tGq}[2]{\Gamma_{\kern-2pt#1}(#2)}  
\newcommand{\tGbil}[3]{\Gamma_{\kern-2pt#1}(#2,#3)}

\newcommand{\Sobolev}[4]{W^{1,2}(0,#1;#2,#3)
\ifthenelse{\equal{#4}{}}{}{\cap
  L^{#4}_{\ifthenelse{\equal{#4}{\infty}}w{}}(0,T;L^{#4}(X,\mm))}}

\newcommand{\GGamma}{\mathbf\Gamma}

\renewcommand{\C}{\mathsf{Ch}}

\def\XXint#1#2#3{{\setbox0=\hbox{$#1{#2#3}{\int}$}
\vcenter{\hbox{$#2#3$}}\kern-.5\wd0}}

\renewcommand{\mm}{\mathfrak m}

\newcommand{\CD}{{\sf CD}}
\newcommand{\RCD}{{\sf RCD}}
\newcommand{\BE}{{\sf BE}}

\newcommand{\weakgrad}[1]{|\rmD #1|_w}                

\renewcommand{\C}{{\sf Ch}_*}

\newcommand{\dom}{\rmD}

\newcommand{\Emeas}[2]{\Gamma^\star_{2,#1}[#2]}
\newcommand{\emeas}[2]{\gamma_{2,#1}[#2]}

\newcommand{\ebilmeas}[3]{\gamma_{2,#1}[#2,#3]}

\renewcommand{\tGq}[1]{\tilde \Gamma(#1)}
\renewcommand{\tGbil}[2]{\tilde \Gamma(#1,#2)}

\renewcommand{\D}{\dom_{L^2}(\Delta)}
\newcommand{\Dinfty}{\D\cap L^\infty(X,\mm)}
\newcommand{\Vinfty}{\V\cap L^\infty(X,\mm)}

\renewcommand{\C}{{\sf Ch}}

\newcommand{\DG}{\dom_\V(\Delta)\cap \GLip}
\newcommand{\GLip}{\L_\Gamma}
\newcommand{\XX}{{\bf X}}

\title{On the Bakry-\'Emery condition, the gradient estimates and the
Local-to-Global property of $\RCD^*(K,N)$ metric measure spaces}
\begin{document}

\author{Luigi Ambrosio\
   \thanks{Scuola Normale Superiore, Pisa, \textsf{l.ambrosio@sns.it}}
   \and
      Andrea Mondino\
    \thanks{ETH, Zurich, \textsf{andrea.mondino@math.ethz.ch}}
      \and
   Giuseppe Savar\'e\
   \thanks{University of Pavia, \textsf{giuseppe.savare@unipv.it}}
   }

\maketitle

\begin{abstract}
We prove higher summability and regularity of
$\Gq f$ for functions $f$ in spaces satisfying the Bakry-\'Emery condition
$\BE(K,\infty)$. 

As a byproduct, we obtain various equivalent weak formulations of 
$\BE(K,N)$ and we prove 
the Local-to-Global property of the $\RCD^*(K,N)$ condition in
locally compact metric measure spaces $(X,\sfd,\mm)$, 
without assuming a priori the non-branching condition on the metric space. 
\end{abstract}

\tableofcontents

\section{Introduction}

\paragraph{\em Curvature-dimension conditions for metric-measure spaces.}
The theory of synthetic Ricci lower bounds has been so far developed
along two lines: 
the \textsc{Bakry-\'Emery} approach \cite{Bakry-Emery84},
see also \cite{Bakry06,Bakry-Ledoux06}, uses the formalism of Dirichlet forms (and the heat flow associated with the Dirichlet form) and it
is based on the so-called $\BE(K,N)$ condition, 
formally expressed in
differential terms by
\begin{equation}
  \label{eq:48}
  \Gamma_2(f)\ge
K\,\Gq f+\frac 1N (\Delta f)^2,\quad
\text{where}\quad
\Gamma_2(f):=\frac 12 \Delta \Gq f-\Gbil f{\Delta f}.
\end{equation}
Here $\Gamma$ is the 
\emph{Carr\'e du Champ} representing the energy density
of a strongly local Dirichlet form
\begin{equation}
  \label{eq:15}
  \cE(f,g):=\int_X \Gbil fg\,\d\mm\quad f,\,g\in \dom(\cE)\subset L^2(X,\mm),
\end{equation}
and $\Delta$ is the associated selfadjoint linear operator in the
Lebesgue space
$L^2(X,\mm)$ (see e.g.~\cite{Bouleau-Hirsch91,Fukushima-Oshima-Takeda11}). A fundamental example is of course given by Euclidean measure
spaces endowed with the classical Dirichlet energy.

The more recent approach of \textsc{Lott-Villani} \cite{Lott-Villani09} and \textsc{Sturm} \cite{Sturm06I,Sturm06II}, based on the so-called $\CD(K,N)$ condition, 
makes sense for  metric measure spaces and it is based on convexity inequalities fulfilled by suitable ``entropies'' along geodesics for the Wasserstein distance.
In the case $N<\infty$, also the more recent variant $\CD^*(K,N)$, see \cite{Bacher-Sturm10}, should be considered,
which provides an (a priori) weaker condition when $K\neq 0$.

Since these theories formalize ``local'' conditions (namely the lower bound on the Ricci tensor and the upper bound on the dimension)
with ``nonlocal'' tools, for both theories it is important to ascertain the validity of the so-called Global-to-Local and Local-to-Global properties.
Since this theme has been more investigated on the $\CD(K,N)$ side, we confine our discussion to this theory, although the equivalence 
results that we shall mention later on could be used to read some results also from the $\BE(K,N)$ side. Typically, the Global-to-Local property
requires some strong convexity property (either of the entropy or of the subdomain under consideration), see for instance \cite[Proposition~30.1]{Villani09},
while the Local-to-Global property has been established under the non-branching assumption, first in $\CD(K,\infty)$ spaces \cite{Sturm06I},
then in $\CD(0,N)$ spaces \cite[Theorem~30.37]{Villani09} and eventually in $\CD^*(K,N)$ spaces \cite{Bacher-Sturm10} 
(see also \cite{Cavalletti} for recent progress on the globalization for $\CD(K,N)$). On the other hand, since
the \hbox{non-branching} assumption is unstable under Gromov-Hausdorff convergence, it is desirable to have stronger axiomatizations
of the $\CD(K,N)$ theory which retain stability and Local-to-Global properties and do not involve the non-branching
assumption. Actually, some results of the $\CD(K,N)$ theory initially proved under the \hbox{non-branching} assumption have been recently 
proved by \textsc{Rajala} without making use of this assumption \cite{Rajala12a}, \cite{Rajala12b}. But, a recent remarkable paper by
 the same author \cite{Rajala13} provides for all $K\in\R$ and $N\geq 1$ a (highly branching) compact metric measure space satisfying 
$\CD^*(0,4)=\CD(0,4)$ locally, but not $\CD(K,\infty)$ (and therefore not even $\CD(K,N)$).
  
\paragraph{\em The case of $\RCD^*(K,N)$ spaces.} 
The main goal of this paper is to prove the Local-to-Global property
in the class of locally compact 
\emph{Riemannian} metric measure spaces $\RCD^*(K,N)$
independently of non-branching assumptions. 
Local compactness is in fact always true in the finite
dimensional case $N<\infty$ \cite[Corollary 2.4]{Sturm06II}, so that it is a restrictive assumption
only in the case $N=\infty$. 

The $\RCD$ axiomatization can be obtained from the $\CD$ one by
just adding the requirement that the metric measure structure is infinitesimally Hilbertian (an assumption suggested by \textsc{Cheeger-Colding} in
\cite[Appendix~2]{Cheeger-Colding97I}); formally, this translates into the assumption that the so-called \textsc{Cheeger}'s energy $\C$ is a quadratic form. 

The class of $\RCD(K,\infty)$ has been introduced in \cite{AGS11b} (and then improved in \cite{AGMR12}), 
while its dimensional counterparts
$\RCD^*(K,N)$ have been studied in the more recent papers \cite{Erbar-Kuwada-Sturm13}, \cite{Ambrosio-Mondino-Savare13}.
Since in the $\RCD$ spaces $\C$ can also be viewed as a Dirichlet form, a more precise connection between the $\RCD^*$ and
the $\BE$ sides of the theory is possible and can indeed be established: without entering here in too many technical details, we just mention that
in \cite{AGS11b} it was proved that $\BE(K,\infty)$ holds for $\RCD(K,\infty)$ spaces, while the implication from $\BE(K,\infty)$
to $\RCD(K,\infty)$ has been established in \cite{AGS12} under mild regularity assumptions on the metric measure structure; 
the dimensional counterparts of this equivalence are given in
\cite{Erbar-Kuwada-Sturm13}, \cite{Ambrosio-Mondino-Savare13}.

Using these connections and partitions of unity, we
can read the local $\RCD$ property as a local $\BE$ property and then use partitions of unity to globalize it; eventually
we use the equivalence in the converse direction to obtain the $\RCD$ property globally.

\paragraph{\em Integral formulation of $\BE(K,\infty)$ and gradient estimates.} 
Even if the classical differential formulation \eqref{eq:48} of
the Bakry-\'Emery condition is clearly local,
the weak-integral $\BE(K,N)$ condition introduced
in \cite{AGS12} has a global character: 
the corresponding $\Gamma_2$
tensor  also involves
a test function $\varphi$ in the multilinear form
\begin{equation}
  \label{eq:40}
  \mathbf \Gamma_2(f;\varphi):=\int_X \Big(\frac
  12\Gamma(f)\Delta\varphi-\Gbil f{\Delta f}\varphi\Big)\,\d\mm,
  \quad
  f,\,\Delta f\in \dom(\cE), \ \varphi,\Delta
\varphi\in L^\infty(X,\mm),
\end{equation}
and the resulting $\BE(K,N)$ condition
\begin{equation}
  \label{eq:42}
  \GGamma_2(f;\varphi)\ge \int_X\Big(K\,\Gq f+\frac 1N (\Delta
  f)^2\Big)\varphi\,\d\mm
  \quad\forevery \varphi\ge 0,
\end{equation}
is thus of global type and involves test functions $\varphi$ which
belong to the domain of $\Delta$ in $L^\infty$.
The formulation based on \eqref{eq:40} and \eqref{eq:42} is carefully adapted to 
deal with the lowest regularity
and
summability properties of $f,\,\varphi$, that should be both sufficient
to give sense to the $\GGamma_2$ tensor and invariant
with respect to the action of the Markov semigroup. The latter is a crucial requirement that
is intrinsically global, not satisfied by the
 stronger differential formulation as in \eqref{eq:48} which 
would impose $\Gq f\in \dom(\Delta)$.
In fact, the typical approach requiring
the existence of an algebra of sufficiently smooth functions
where all the relevant computations can be carried on, is quite useful
to deal with many concrete examples but it does not seem to
be well adapted to the non-smooth framework of general metric measure spaces.

Therefore finding useful localizations of $\BE(K,N)$ is not a trivial
issue,
since it involves the summability of $\Delta \varphi$ and the
regularity of $\Delta f$ and of $\Gq f$.
Recall that a product with a test function $\nchi$
affects $\Delta f$ through the Leibniz formula
\begin{equation}
  \label{eq:43}
  \Delta(f\nchi)=\nchi\Delta f+f\Delta\nchi+2\Gbil f\nchi,
\end{equation}
thus showing the importance to secure existence of good classes of cutoff functions 
$\nchi$ with $\nchi,\Gq \nchi,\,\Delta\nchi\in L^\infty(X,\mm)$ 
(a problem addressed in Lemma~\ref{lem:cutoff}) and
general conditions ensuring 
$\Gq f \in \dom (\cE).$ Similar problems arise with the chain rule
\begin{equation}
  \label{eq:44}
  \Delta \eta(f)=\eta'(f)\Delta f+\eta''(f)\Gq f,\quad \eta\in \rmC^2(\R).
\end{equation}
It is therefore natural to investigate 
higher 
integrability and regularity 
properties of $\Gamma(f)$ (see Theorems~\ref{thm:interpolation} and
\ref{thm:crucial}) 
that  are interesting by themselves and will also play a role
in our forthcoming paper \cite{Ambrosio-Mondino-Savare13}.

Improving some results of \cite{Savare13}, in this paper we show that in $\BE(K,\infty)$ spaces
functions $f\in L^2\cap L^\infty(X,\mm)$ with $\Delta f\in L^2(X,\mm)$
satisfy
the extra integrability property for $\Gq f$
\begin{equation}
  \label{eq:45.1}
  \Gq f\in L^2(X,\mm),\quad 
  \int_X \Gq f^2 \,\d\mm\le A_{K} \|f\|_{L^\infty(X,\mm)}^2\int_X \big(f-\Delta
  f\big)^2\,\d\mm.
\end{equation}
In addition, if $f$ and $\Delta f$ belong to $L^4(X,\mm)$ then $\Gq f$ belongs to the domain of
the Dirichlet form $\cE$ and  satisfies
\begin{equation}
  \label{eq:45}
  \Gq f\in \dom(\cE),\quad 
  \int_X \Big(\Gq f^2 +\Gq {\Gq f}\Big)\,\d\mm\le B_{K}\int_X \big(f-\Delta
  f\big)^4\,\d\mm.
\end{equation}
The constants $A_K, \,B_K$ in the previous estimates depend only on $K$.
These properties allow for simpler formulation of \eqref{eq:42} 
and are the starting points for studying its localization, since 
we will show that the same estimates hold even if $X$ is covered 
by a collection of spaces satisfying $\BE(K,\infty)$.

\paragraph{\em Plan of the paper.}
The paper is organized as follows: in Section~\ref{sec:BE} we will work in the general framework of Dirichlet spaces, 
without assuming that the Dirichlet form $\cE$ is induced by the Cheeger energy and 
actually avoiding any reference to a metric structure 
(so that the role of modulus of the weak gradient is played by  $\sqrt{\Gamma(f)}$); more precisely we just assume that $(X,\tau)$ is a Polish topological 
space endowed with a $\sigma$-finite reference Borel measure $\mm$ and a strongly local and symmetric Dirichlet form $\cE$ on $L^2(X,\mm)$ 
enjoying a \emph{Carr\'e du Champ} $\Gamma:D(\cE)\times D(\cE)\to L^1(X,\mm)$ and a
$\Gamma$-calculus  (see e.g.\ \cite[\S~2]{AGS12}).  In this framework, under the $\BE(K,\infty)$ condition,  we establish 
useful higher integrability properties for $\Gamma(f)$
as \eqref{eq:45.1} by an interpolation argument (Section~\ref{sec:interpolation}) 
and the extra-regularity property \eqref{eq:45} (Section~\ref{sec:further-regularity}).
These properties will be used in \cite{Ambrosio-Mondino-Savare13} 
and in the second part of the paper to prove the Local-to-Global property. 
Still in the same framework, in Section~\ref{sec:equiv} we provide equivalent formulations and implications of the $\BE(K,N)$ property that
play a role in this and in the companion paper \cite{Ambrosio-Mondino-Savare13}. 

In the second part, composed by Sections~\ref{sec:LTGComp} and
Section~\ref{sec:RCD}, we will work instead with metric measure spaces
and we will use the previous estimates to prove the Local-to-Global
property. 
In Section~\ref{sec:LTGComp} we discuss basic localization
properties of gradients and Laplacians and show how curvature lower
bounds can be used to obtain existence of cutoff functions with
bounded Laplacian. 

In Section~\ref{sec:RCD} we recall the precise definitions of
$\RCD^*(K,N)$ spaces, the equivalence results with $\BE(K,N)$,
and we carry on the proof of the Local-to-Global property 
in case the space $(X,\sfd,\mm)$ is locally compact.

\smallskip
\noindent {\bf Acknowledgement.} The authors acknowledge the support
of the ERC ADG GeMeThNES and warmly thank Tapio Rajala for his comments
on a preliminary version of this paper.

\section{Notation, preliminaries and the Bakry-\'Emery condition} 
\label{sec:BE}
In this section we will recall the basic assumptions related to the
Bakry-\'Emery condition. 
\paragraph{\em Strongly local Dirichlet forms and $\Gamma$-calculus.}
The natural setting is provided
by a Polish topological space $(X,\tau)$ endowed with
a $\sigma$-finite reference Borel measure $\mm$ with full support
(i.e. $\mathrm{supp}(\mm)=X$) and 
\begin{equation}
 \begin{gathered}
    \text{a strongly local, symmetric Dirichlet form $\cE$ on
      $L^2(X,\mm)$ enjoying}\\
    \text{a \emph{Carr\'e du Champ}
    $\Gamma:\rmD(\cE)\times \rmD(\cE)\to L^1(X,\mm)$ and}\\
  \text{generating
  a mass-preserving Markov semigroup $(\sfP_t)_{t\ge0}$ on $L^2(X,\mm)$,}
\end{gathered}
\label{eq:18}
\end{equation}
(see e.g.\ \cite[\S 2]{AGS12}).
None of the estimates we are discussing in this  section really needs
an underlying compatible metric structure, as the one discussed in 
\cite[\S 3]{AGS12}. 
We refer to \cite[\S 2]{AGS12} for the basic notation and assumptions;
in any case, we will apply all the results to the case of the Cheeger
energy (thus assumed to be quadratic) 
of the metric measure space $(X,\sfd,\mm)$ and we will use the calculus properties of 
the Dirichlet form that are related to the $\Gamma$-formalism.

In the following we call $\V$ the Hilbert space made by
$\dom(\cE)\subset L^2(X,\mm)$ endowed with the scalar product 
$$(f,g)_{\V}:=\int_X \Big(fg+ \Gamma(f,g)\Big) \, \d \mm.$$ 
The Laplace operator $-\Delta:\V\to \V'$
and its perturbation $-\Delta_\lambda$ are respectively defined as
\begin{equation}
  \label{eq:defDelta}
  \langle -\Delta f,g\rangle:=\cE(f,g),\quad
  -\Delta_\lambda f=\lambda f-\Delta f\quad\forevery f,\,g\in \V.
\end{equation}
The operator $\Delta$ is the generator of the Markov semigroup
$(\sfP_t)_{t\ge0}$ and its realization in $L^2(X,\mm)$ is an
unbounded selfadjoint nonnegative operator with domain $\D$.

We will denote by 
\begin{equation}
  \label{eq:38}
  \dom_{L^p}(\Delta):= \big\{f\in \V\cap L^p(X,\mm):\Delta f\in
  L^2\cap L^p(X,\mm)\Big\},\quad p\in [1,\infty],
\end{equation}
its domain as unbounded operator
in $L^p(X,\mm)$, endowed with the norm $\|f\|_\V+\|f-\Delta
f\|_{L^2\cap L^p}$. This choice of the norm is justified by the inequalities 
  \begin{equation}
    \label{eq:39}
    \lambda f-\Delta f=g\in L^p(X,\mm)\quad\Longrightarrow\quad
    \lambda \|f\|_{L^p}\le \|g\|_{L^p},\quad
    \|\Delta f\|_{L^p}\le 2\|g\|_{L^p}=
    2\|\Delta_\lambda f\|_{L^p}
  \end{equation}
for all $\lambda\geq 0$. In turn, the implication \eqref{eq:39} follows by the fact that the resolvents $\lambda
(\lambda-\Delta)^{-1}$, $\lambda>0$, associated to a Dirichlet form
are sub-Markovian (see e.g.~\cite[Def.~4.1 and Thm.~4.4]{Ma-Rockner92})
and therefore contractive in every $L^p(X,\mm)$.

\paragraph{\em The $\Gamma_2$ tensor and the Bakry-\'Emery condition.} We 
introduce the multilinear form $\GGamma_2$ given by
\begin{equation}
\begin{aligned}
  \GGamma_2(f,g;\varphi):=&
  \frac 12\int_X \Big(\Gbil fg\, {\Delta \varphi}-
  \big(\Gbil f{\Delta g}+\Gbil g{\Delta f}\big)\varphi\Big)\,\d\mm
  \quad
  (f,g,\varphi)\in \dom(\mathbf\Gamma_2),
\end{aligned}\label{eq:22}
\end{equation}
where $\dom(\mathbf{\Gamma}_2):=\dom_\V(\Delta)\times \dom_\V(\Delta) \times
  \dom_{L^\infty}(\Delta),$ and 
$\dom_\V(\Delta)=\big\{f\in \V:
  \Delta f\in \V\big\}$.
When $f=g$ we also set
\begin{equation}
  \label{eq:80} 
  \GGamma_2(f;\varphi):=\GGamma_2(f,f;\varphi)=
  \int_X \Big(\frac 12\Gq f\, {\Delta \varphi}-
  \Gbil f{\Delta f}\varphi\Big)\,\d\mm,
\end{equation}
so that 
\begin{equation}
  \label{eq:46}
  \GGamma_2(f,g;\varphi)=  \frac14\GGamma_2(f+g;\varphi)-\frac 14 \GGamma_2(f-g;\varphi).
\end{equation}

The multilinear form $\mathbf\Gamma_2$ provides a weak version 
(inspired by \cite{Bakry06,Bakry-Ledoux06}) of 
the Bakry-\'Emery condition \cite{Bakry-Emery84,Bakry92}. In the sequel, given $f:X\to\R$, we denote by 
$\supp(f)$ the smallest closed set $C\subset X$ such that $f=0$ $\mm$-a.e. in $X\setminus C$; this way, the definition
of support is independent of modifications of $f$ in $\mm$-negligible
sets.

\begin{definition}[Bakry-\'Emery condition]\label{def:BE}
  Let $K\in \R$, $N\in [1,\infty]$, and $\nu:=\frac 1N\in [0,1]$.
  We say that the strongly local Dirichlet form $\cE$ satisfies the
  $\BE(K,N)$
  condition, 
  if it admits a Carr\'e du Champ $\Gamma:\dom(\cE)\times \dom(\cE)\to L^1(X,\mm)$ and
  for every $(f,\varphi)\in \dom_\V(\Delta)\times \dom_{L^\infty}(\Delta)$
  with $\varphi\ge0$ there holds
  \begin{equation}
  \label{eq:9}
  \GGamma_2(f;\varphi)\ge K\int_X \Gq f\,\varphi\,\d\mm+\nu\int_X
  (\Delta f)^2\varphi\,\d\mm.
\end{equation}
  We say that $\cE$ satisfies the $\BE_{{\rm loc}}(K,N)$ condition if \eqref{eq:9} holds for all
  $(f,\varphi)\in \dom_\V(\Delta)\times \dom_{L^\infty}(\Delta)$
  with $\varphi\ge0$ compactly supported. 
\end{definition}
 
 \begin{remark}[On the global character of the $\BE(K,N)$ condition]\label{rem:BEglobal}
 \upshape
 {\rm Notice 
   that $\BE(K,N)$ has a global nature, related to the fact that an integration by parts is understood
 in the weak formulation \eqref{eq:80} of Bochner's inequality; for this reason, even the issue of the Global-to-Local property is delicate in this framework,
 since the passage to a smaller open set $U\subset X$ changes the Dirichlet form and the action of the Laplacian operator
 (unless one deals with functions compactly supported
 in $U$, compare with Remark~\ref{rem:locality}). As a matter of fact, the localization seems to involve some ``metric"
 assumption on $U$, relative to the distance $\sfd_{\cE}$ induced by $\cE$, see Proposition~\ref{prop:localK}(c). For this reason, in the discussion
 of the Local-to-Global and Global-to-Local properties, we will deal
 with metric measure spaces under a metric version of the $\BE(K,N)$ condition
 (see Definition~\ref{def:topo-BE}), although
 the equivalence results of \cite{AGS12}, \cite{Erbar-Kuwada-Sturm13}, \cite{Ambrosio-Mondino-Savare13} could be used
 to translate back the result to the $\BE$ formalism.
 }\fr\end{remark}

\begin{remark}[Pointwise gradient estimates for the heat flow under $\BE(K,\infty)$]
  \upshape
  When $N=\infty$, condition \eqref{eq:9} is in fact equivalent (see \cite[Corollary~2.3]{AGS12}) to either of the following
  pointwise gradient estimates
  \begin{equation}
    \label{eq:77}
    \Gq{\sfP_t f}\le \rme^{-2K t}\,\sfP_t\big(\Gq f\big)\quad
    \text{$\mm$-a.e.\ in $X$, for every }f\in \V,
  \end{equation}
  \begin{equation}
    \label{eq:77bis}
    2\rmI_{2K}(t) \Gq{\sfP_t f}\le \sfP_t{f^2}-\big(\sfP_t f\big)^2
    \quad \mm\text{-a.e.\
      in }X,\quad
    \forevery
    t>0,\ f\in L^2(X,\mm),
  \end{equation}
  where $\rmI_K$ denotes the real function
  \begin{displaymath}
    \rmI_{K}(t):=\int_0^t \rme^{K r}\,\d r=
    \begin{cases}
      \frac 1K(\rme^{K t}-1)&\text{if }K\neq 0,\\
      t&\text{if }K=0.  \qquad\blacksquare
    \end{cases}   
  \end{displaymath}
 \end{remark}

\section{Interpolation estimates: extra integrability of $\Gamma(f)$ }
\label{sec:interpolation}
Let us now consider the semigroup $(\sfP^\lambda_t)_{t\ge0}$
generated by the
operator $\Delta_\lambda:=\Delta-\lambda$,
$\lambda\ge0$,
\begin{equation}
  \label{eq:78}
  \sfP^\lambda_t f:=\exp(t\Delta_\lambda)f=\rme^{-\lambda t}\exp(t\Delta)f=
  \rme^{-\lambda t}\,\sfP_t f.
\end{equation}
Since $\sqrt{\rmI_{2K}(t)}\ge \sqrt t \,\rme^{Kt}\quad\text{if }K\le 0$,
choosing $\lambda\ge K_-$ and $p\ge 2$, $\BE(K,\infty)$ and the contractivity of $\sfP_t$ in $L^p$ yield by \eqref{eq:77bis}
\begin{equation}
  \label{eq:79}
  \big\|\Gq {\sfP^\lambda_t f}^{1/2}\big\|_{L^{p}(X,\mm)}\le \frac 1{\sqrt
    {2t}}\|f\|_{L^p(X,\mm)}\quad \forevery f\in L^p(X,\mm),\,\,t>0.
\end{equation}
We prove now a useful estimate for $\Gq f$ in 
$L^p(X,\mm)$  when $\Delta f\in L^p(X,\mm)$ 
and $f$ is bounded. For the sake of simplicity, we will only focus
on the cases $p=2$ and $p=\infty$, 
that will also play a role in \cite{Ambrosio-Mondino-Savare13}.
Analogous results (in $X=\R^d$)
when only a one-sided bound on $f$ is available have been proved with
completely different proofs in
\cite{Lions-Villani95,Gianazza-Savare-Toscani09}.

\begin{theorem}[Gradient interpolation]
  \label{thm:interpolation}
  Assume $\BE(K,\infty)$, let $\lambda\ge K_-$,  $p\in \{2,\infty\}$,
  and $f\in L^2\cap L^\infty(X,\mm)$ with $\Delta f\in
  L^p(X,\mm)$.
  Then $\Gq f\in L^{p}(X,\mm)$ 
  and
  \begin{equation}
    \label{eq:87}
    \big\|\Gq f\|_{L^p(X,\mm)}\le
    c\|f\|_{L^\infty(X,\mm)}\,\|\Delta_\lambda f\|_{L^p(X,\mm)}
  \end{equation}
  for a universal constant $c$ independent of $\lambda,X,\mm,f$
  ($c=\sqrt{2\pi}$ when $p=\infty$).
  
  Moreover, if $f_n\in \Dinfty$ with $\sup_n \|f_n\|_{L^\infty(X,\mm)}<\infty$ and 
  $f_n\to f$ strongly in $\D$, then $\Gq {f_n}\to \Gq f$ 
  and $\Gq {f_n-f}\to0$ strongly in $L^2(X,\mm)$.
\end{theorem}
\begin{proof}
  Let us first consider the case $p=\infty$ (here we follow the argument
  of \cite[Prop.~3.6]{Coulhon-Sikara10}):
  recalling the identity
  \begin{displaymath}
    f=\int_0^\infty \rme^{-s}\,
    \sfR_{ts} (f-tAf)\,\d s
  \end{displaymath}
  valid for nonnegative selfadjoint semigroups $\sfR$ with infinitesimal generator $A$,
  by applying \eqref{eq:79} with $p=\infty$ to $f_t:=f-t\Delta_\lambda f$ we get
  \begin{align*}
    \Big\|\Gq{ f}^{1/2}\Big\|_{L^\infty(X,\mm)}
    &\le \|f_t\|_{L^\infty(X,\mm)}\,\int_0^\infty \rme^{-s}(2
    ts)^{-1/2}\,\d s=
    \sqrt{\frac \pi {2t}}\,\|f_t\|_{L^\infty(X,\mm)}
    \\&\le 
    \sqrt{\frac \pi {2t}}\,\Big(\|f\|_{L^\infty(X,\mm)}+t
    \|\Delta_\lambda f\|_{L^\infty(X,\mm)}\Big).
  \end{align*}
  Choosing $t=\|f\|_\infty\,\|\Delta_\lambda f\|_\infty ^{-1}$ 
  we obtain \eqref{eq:87}.

 In order to prove the formula \eqref{eq:79} in the case $p=2$,
%
by an elementary approximation, suffices to show the inequality under the additional assumption
$f\in\V$, $\Delta f\in\V$. We use the Leibniz formula
$$
\frac{\d}{\d t}\Gq {\sfP^\lambda_t f }=2\Gbil{\sfP^\lambda_t f}{\frac{\d}{\dt}\sfP^\lambda_t f}
$$
to get
\begin{equation}\label{eq:11sep}
\left|\frac{\d}{\d t} \Gq {\sfP^\lambda_t f }^{1/2}\right|^2 = \frac{\left(\Gamma(\sfP^\lambda_t f, \frac {\d}{\d t} \sfP^\lambda_t f )\right)^2}
{\Gamma(\sfP^\lambda_t f)} \leq \Gq{\frac \d{\dt} \sfP^\lambda_t f}.
\end{equation}
Commuting $\Delta$ with $\frac{\d}{\dt}$, using the
identity $\frac\d{\d t} \sfP^\lambda_t f=\Delta_\lambda \sfP^\lambda_t f$ together with the fact that
$\lambda\geq 0$ we get
\begin{eqnarray}
  \label{eq:85}
  \int_0^\infty\int_X \Gq{\frac \d{\dt} \sfP^\lambda_t f}\,\d\mm\,\dt&=&
  -\int_0^\infty\int_X(\frac \d{\dt}\sfP^\lambda_t f)(\frac \d{\dt}\Delta\sfP^\lambda_t f)\,\d\mm\,\dt\\&=&
  -\int_0^\infty\int_X(\frac \d{\dt}\sfP^\lambda_t f)(\frac \d{\dt}\Delta_\lambda\sfP^\lambda_t f)\,\d\mm\,\dt-
  \lambda\int_0^\infty\int_X\bigl(\frac \d{\dt}\sfP^\lambda_t f\bigr)^2\,\d\mm\,\dt\nonumber\\&\leq&
  -\frac 12\int_0^\infty\int_X\frac \d{\dt}|\Delta_\lambda\sfP^\lambda_t f|^2\,\d\mm\,\dt
  \leq \frac 12 \int_X \big(\Delta_\lambda f\big)^2\,\d\mm.\nonumber
\end{eqnarray}
Setting $g_t:=\Gq {\sfP^\lambda_t f }^{1/2}$, inserting \eqref{eq:11sep} into \eqref{eq:85} it follows that
\begin{equation}
  \label{eq:86}
  \int_0^\infty \|t^{1/2}\frac \d{\dt} g_t\|_{L^2(X,\mm)}^2 \,\frac{\d t}t\leq
  \frac 12\|\Delta_\lambda f\|_{L^2(X,\mm)}^2.
\end{equation}

According to the J.L.~Lions Trace interpolation method
(here we follow the notation of \cite[1.8.1]{Triebel95}),
the estimates \eqref{eq:79} and 
\eqref{eq:86} show that 
$g_t$ belongs to the weighted functional space
$V_1(\infty,\frac 12,L^\infty(X,\mm);2,1/2,L^2(X,\mm))$, 
so that its trace at $t=0$ belongs to 
the $K$-interpolation space
\begin{displaymath}
  g_0\in (L^\infty(X,\mm),L^2(X,\mm))_{\theta,p}\quad\text{with}\quad
  \theta=\frac12,\quad
  p=4,
\end{displaymath}
with
\begin{displaymath}
  \|g_0\|_{(L^\infty(X,\mm),L^2(X,\mm))_{\theta,p}}\le c
  \|f\|_{L^\infty(X,\mm)}^{1/2}
  \|\Delta_\lambda f\|_{L^2(X,\mm)}^{1/2}.
\end{displaymath}
Since $g_0=\Gq f^{1/2}$ we get \eqref{eq:87} also in the case $p=2$.

Let us now consider the last statement. The fact that $\Gq{f_n-f}\to 0$ strongly in $L^2(X,\mm)$ 
follows immediately from the interpolation inequality \eqref{eq:87} by replacing
$f$ with $f_n-f$ and observing that $\Delta_\lambda(f_n-f)\to0$ strongly in $L^2(X,\mm)$ 
since $f_n\to f$ in $\D$. Recalling that 
$$\big|\Gq{f_n}-\Gq f\big|^2=
\big|\Gbil{f_n-f}{f_n+f}\big|^2\le 
\Gq{f_n-f}\Gq{f_n+f}$$
we also get $\Gq{f_n}\to\Gq{f}$ strongly in $L^2(X,\mm)$.
\end{proof}

\section{Equivalent formulations of $\BE(K,N)$}
\label{sec:equiv}
Let $f\in \dom_\V(\Delta)$ and $\varphi\in \dom_{L^\infty}(\Delta)$ and let us consider the expression
\eqref{eq:80} of $\GGamma_2(f;\varphi)$; under the additional assumption $f\in \dom_{L^\infty}(\Delta)$,
by ``integrating by parts" the term $\Gbil f{\Delta f}$ 
it is possible to write $\GGamma_2(f;\varphi)$ in a different form:
\begin{lemma}\label{lem:GGamma2tilde}
  If $f\in \dom_\V(\Delta)\cap \dom_{L^\infty}(\Delta)$ and $\varphi\in \dom_{L^\infty}(\Delta)$ 
  then
  \begin{equation}
    \label{eq:30}
    \GGamma_2(f;\varphi)=\int_X \Big(\frac12 \Gq f\Delta\varphi+
    \Delta f\,\Gbil f\varphi+\varphi(\Delta f)^2\Big)\,\d\mm.
  \end{equation}
\end{lemma}
\begin{proof}
  Starting from the Leibniz formula
  \begin{displaymath}
    \Gbil f g\varphi=\Gbil f{g\varphi}-\Gbil f\varphi g\quad
    \forevery f\in \V,\ g,\varphi\in \Vinfty,
  \end{displaymath}
  if $f\in \D$ one has by integration
  \begin{displaymath}
    \int_X \Gbil f g\varphi\,\d\mm=-\int_X  \Big( g\varphi\,\Delta f +\Gbil f\varphi g\Big)\,\d\mm.
  \end{displaymath}
  Choosing $g=\Delta f\in \Vinfty$ we immediately see that 
  \eqref{eq:80} yields \eqref{eq:30}.
\end{proof}
Recalling the polarization identity \eqref{eq:46},
if $f,\,g\in  \dom_\V(\Delta)\cap \dom_{L^\infty}(\Delta)$ and  $\varphi\in \dom_{L^\infty}(\Delta)$, 
from \eqref{eq:30} we also get
\begin{equation}
  \label{eq:49}
  \GGamma_2(f,g;\varphi)=\frac 12 \int_X \Big(\Gbil fg\Delta\varphi+
  \Delta f\,\Gbil g\varphi+
  \Delta g\,\Gbil f\varphi+2\varphi \, \Delta f\,\Delta g\Big)\,\d\mm.
\end{equation}

In the passage from \eqref{eq:22} to \eqref{eq:30} (or, equivalently, \eqref{eq:49}) we used the additional regularity
assumption $f\in \dom_{L^\infty}(\Delta)$; therefore the following approximation result will be useful in the verification of the 
$\BE(K,N)$ property.

\begin{lemma}[Approximation of $\dom_\V(\Delta)$ by $\dom_\V(\Delta) \cap   \dom_{L^\infty}(\Delta)$]\label{lem:Approximation}
Let $f \in \dom_\V(\Delta)$. Then there exist $f_n\in  \dom_\V(\Delta) \cap   \dom_{L^\infty}(\Delta)$ such that 
$$f_n \to f \text{ in } \V \quad\text{and}\quad \Delta f_n \to \Delta f \text{ in } \V. $$ 
In particular, $\BE(K,N)$ holds if and only if \eqref{eq:9} holds for all $f\in  \dom_\V(\Delta) \cap   \dom_{L^\infty}(\Delta)$
and all $\varphi\in \dom_{L^\infty}(\Delta)$ nonnegative.
\end{lemma} 

\begin{proof}
Let $f \in \dom_\V(\Delta)$ and define $h\in \V$  by $h:= f-\Delta f$. Consider the truncated functions 
\begin{equation}
h_n:=\max\{\min\{h,n \}, -n\}.\label{eq:6}
\end{equation}

Clearly $h_n \in \V\cap L^\infty(X,\mm)$ and $h_n \to h$ in $\V$. Let $f_n\in \V$ be the unique (weak)  solution of
\begin{equation}\label{eq:fn}
f_n-\Delta f_n=h_n. 
\end{equation}
Of course $f_n\to f \in \V$ and  the variational maximum principle
implies
that $|f_n|\leq |h_n|{\GGG\le} n$ $\mm$-a.e. in $X$; let us briefly recall the 
argument, well known in literature as \emph{Stampacchia's truncation}. The  solution $f_n$ of \eqref{eq:fn} is the unique minimum point of 
the strictly convex functional 
$$\V\ni g\mapsto \frac{1}{2}\int_X |g-h_n|^2 \, \d\mm + \frac{1}{2}\cE(g,g);$$
replacing $f_n$ by the truncated function $\tilde{f}_n:= \max\{\min\{f_n,n \}, -n\}\in\V$ neither of the integrals above increase. 
It follows that $f_n=\tilde{f}_n$ $\mm$-a.e.  in $X$, as desired.

We conclude by observing that, since $f_n$ belong to  $\V\cap L^\infty(X,\mm)$, we have  $\Delta f_n= f_n-h_n$ belong to $\V\cap L^\infty(X,\mm)$ as well and, 
since $h_n\to h$ and $f_n\to f$ in $\V$, we get $\Delta f_n \to \Delta f$ in $\V$.  
\end{proof}

Thanks to the improved integrability of $\Gamma$, provided by the $\BE(K,\infty)$ condition, we can now somehow extend the domain of $\GGamma_2(f,g;\varphi)$
to $\Big(\Dinfty\Big)^3$, 
i.e. neither requiring $\Delta f$, $\Delta g$ to be in $\V$ nor requiring $\Delta\varphi$ to be in $L^\infty(X,\mm)$.

\begin{corollary}
  \label{cor:weaker-assumption-G2} If $\BE(K,\infty)$ holds then 
  the right hand side of 
  \eqref{eq:49} makes sense in the space $\Big(\Dinfty\Big)^3$.
  In addition, if $\BE(K,N)$ holds then
  \begin{equation}\label{eq:Reading}
    \int_X \Big(\frac12 \Gq f\Delta\varphi+
    \Delta f\,\Gbil f\varphi+\varphi(\Delta f)^2\Big)\,\d\mm\geq
    K\int_X \Gq {f}\,\varphi\,\d\mm+\nu\int_X (\Delta f)^2\varphi\,\d\mm
  \end{equation}
  is satisfied by every choice of $f,\,\varphi\in \Dinfty$ with $\varphi\ge0$.
\end{corollary}
\begin{proof}
  Notice that the right hand side of
  \eqref{eq:49} makes sense if $f,\,g,\varphi\in \Dinfty$ 
  since $\Gq f,\,\Gq g,\,\Gq \varphi\in L^2(X,\mm)$ by Theorem~\ref{thm:interpolation}.
  Under the assumption $\BE(K,N)$, in order to check \eqref{eq:Reading} we introduce the mollified heat flow
  \begin{equation}
    \label{eq:55}
    \mathfrak H^\eps f:= \frac 1\eps \int_0^\infty \sfP_r f\,\kappa(r/\eps)\,\d r,
  \end{equation}
  where $\kappa\in \rmC^\infty_c(0,\infty)$ is a nonnegative regularization kernel
  with $\int_0^\infty \kappa(r)\,\d r=1$.

  Setting $f^\eps:=\mathfrak H^\eps f,\ \varphi^\eps:=\mathfrak H^\eps\varphi$, 
  it is not difficult to check that if $f,\varphi\in \Dinfty$ then $f^\eps,\varphi^\eps
  \in \dom_{\V}(\Delta)\cap \dom_{L^\infty}(\Delta)$; morever, $\varphi^\eps\ge 0$ if $\varphi\ge0$, so that 
  \eqref{eq:9} and \eqref{eq:30} yield
  \begin{displaymath}
    \int_X \Big(\frac12 \Gq {f^\eps}\Delta\varphi^{\eps}+
    \Delta f^\eps\,\Gbil {f^\eps}{\varphi^\eps}+ (1-\nu)\varphi^\eps(\Delta f^\eps)^2\Big)\,\d\mm
    \ge K\int_X \Gq {f^\eps}\,\varphi^\eps\,\d\mm.
  \end{displaymath}
  Since 
  \begin{displaymath}
    f^\eps\to f,\ \varphi^\eps\to \varphi\quad\text{strongly in }\D,\quad
    \|f^\eps\|_{L^\infty(X,\mm)}\le \|f\|_{L^\infty(X,\mm)},\quad
    \|\varphi^\eps\|_{L^\infty(X,\mm)}\le \|\varphi\|_{L^\infty(X,\mm)}
  \end{displaymath}
  we can apply the continuity properties of $\Gamma$ stated in Theorem \ref{thm:interpolation}
  to pass to the limit in the previous inequality as $\eps\downarrow0$.
\end{proof}

\section{Further regularity for $\Gamma(f)$ in $\BE( K,\infty)$ spaces and the measure-valued 
$\Gamma_2$-tensor}
\label{sec:further-regularity}
\newcommand{\cVrestr}[1]{\cV_{\kern-2pt #1}}
\newcommand{\cVl}[1]{\cV_{\kern-2pt #1}}
\newcommand{\cp}{\mathrm{Cap}}

\subsection{Quasi-regular Dirichlet forms and the measure-valued 
$\Gamma_2$-tensor}
In this section we will assume that 
the Dirichlet form $\cE$ is \emph{quasi-regular}, according to 
\textsc{Ma and R\"ockner}:
we refer to  
\cite[III.2, III.3, IV.3]{Ma-Rockner92} 
(covering the more general case of a
possibly non-symmetric Dirichlet form) 
and
\cite[1.3]{Chen-Fukushima12}
for the precise definition and for the related notions of
$\cE$-polar sets and $\cE$-quasi-continuous functions, 
see also 
the concise account of \cite[\S~2.3]{Savare13}.
Here we just recall that this setting covers the main example
of \emph{regular} Dirichlet forms in \emph{locally compact and
  separable metric spaces}, which is sufficient 
for our main applications in the next sections. 
Still the results presented here, at least in the case $\BE(K,\infty)$,
could be interesting 
in more general situations where $(X,\tau)$ is not locally compact:
this is the reason why we state them in greater generality.
\begin{remark}[Regular Dirichlet forms]
  \label{rem:regular}
  \upshape
  When $(X,\tau)$ is also locally compact, 
  we recall that a Dirichlet form $\cE$ is \emph{regular} if 
  $\rm\dom(\cE)\cap \rmC_{c}(X)$ is dense both in $\rm\dom(\cE)$ (w.r.t.\ the
  $\V$-norm)
  and
  in $\rmC_c(X)$ (w.r.t.~uniform convergence).  
\end{remark}
If $\cE$ is quasi-regular, then 
\cite[Remark 1.3.9(ii)]{Chen-Fukushima12}
\begin{displaymath}
\text{every function $f\in \cV$ admits an $\cE$-quasi-continuous
representative $\tilde f$.}
\end{displaymath}
The function $\tilde f$ is uniquely determined up to a $\cE$-polar set.
We introduce the convex set 
$$\cVl + :=\big\{\varphi\in \cV:\varphi\geq
0\ \mm\text{-a.e. in }X\big\}$$ 
and we  denote by $\cVl {}'$ the 
set of continuous linear functionals $\ell:\V\to\R$, while $\cVl +'$ denotes the 
convex subset of all continuous linear functionals $\ell$ such that
$\langle \ell,\phi\rangle\geq 0$ for all $\phi\in \cVl +$;
we also set $\cVl \pm':=\cVl +'-\cVl +'$.

The next result provides an important characterization
of functionals in $\cVl +'$, that motivates our interest for
quasi-regular Dirichlet forms (see \cite[Ch.~VI,
Prop.~2.1]{Ma-Rockner92}
and also \cite[Ch.~I, \S~9.2]{Bouleau-Hirsch91} in the case of 
a finite measure $\mm(X)<\infty$ for the proof).

\begin{proposition}
\label{prop:Daniell}  
Let us assume that $\cE$ is quasi-regular. 
Then for every $\ell\in \cVl +'$ there exists a unique $\sigma$-finite and 
nonnegative Borel measure $\mu_\ell$ in $X$
such that 
\begin{enumerate}[\rm (1)]
\item 
every $\cE$-polar set 
is $\mu_\ell$-negligible; 
\item for all $f\in \cV$ the $\cE$-q.c.~representative $\tilde f$ belongs to $L^1(X,\mu_\ell)$ and
\begin{equation}
  \label{eq:13}
  \langle \ell,f\rangle=\int_X \tilde f\,\d\mu_\ell;
\end{equation}
\item if $\langle \ell,\varphi\rangle\le M<\infty$ for every $\varphi\in
  \cVl +$ with $\varphi\le 1$ $\mm$-a.e. in $X$,
then $\mu_\ell$ is a finite measure and $\mu_\ell(X)\le M$.
\end{enumerate}
\end{proposition}
We will often identify $\ell\in \cVl+'$ with 
the corresponding measure $\mu_\ell$. 
Notice that if $\ell\in \V'$ and there exists $h\in L^1\cap L^2(X,\mm)$ such
that 
\begin{equation}
  \label{eq:51}
  \langle \ell,\varphi\rangle\ge \int_X h\varphi\,\d\mm\quad \text{for
    every }\varphi\in \cVl +
\end{equation}
then there exists a measure $\mu_+\in \cVl+'$ such that 
$\ell$ can be represented by the signed meassure $\mu_\ell=h\mm+\mu_+$.
When for some $f\in\V$ the functional $\ell=\Delta f$  
can be identified with a signed measure $\mu_\ell$, we
will use the notation $\mu_\ell=\Delta^* f$.

The next result collects a few useful 
properties that have been  proved in
\cite[\S\,3]{Savare13};  we introduce the space 
$\GLip:=\big\{f\in \V:\ f,\,\Gq f\in L^\infty(X,\mm)\big\}$.
\begin{theorem}
  \label{thm:mainquote}
  Let us suppose that $\cE$ 
  satisfies the
  $\BE(K,\infty)$ condition.
  \begin{enumerate}[\rm (1)]
  \item 
    For every
    $f\in \dom_\V(\Delta)\cap \GLip$
    we have $\Gq f\in \V$ with
  \begin{equation}
    \label{eq:G3}
    \cE({\Gq f})\le -\int_X \Big(
    2K\Gq f^2+ 2\Gq f\Gbil f{\Delta f}\Big)\,\d\mm.
  \end{equation}
  \item
  $\dom_\V(\Delta)\cap \GLip$ is an algebra (i.e. closed w.r.t.~pointwise
  multiplication)
  and, more generally,
  if $\ff=(f_i)_{i=1}^n\in (\DG)^n$ then $\Phi(\ff)\in \DG$ 
  for every smooth function $\Phi:\R^n\to \R$ with $\Phi(0)=0$.
  \item If $\cE$ is also quasi-regular and $f\in \dom_\V(\Delta)\cap \GLip$, then $\Delta \Gq f$ 
  can be represented by a signed measure vanishing on $\cE$-polar sets and, defining
  \begin{equation}
    \label{eq:50}
    \Gamma^*_{2,K} [f]:=\frac 12 \Delta^\star \Gq f-\Big(\Gbil f{\Delta f}+K\Gq f\Big)\mm,
  \end{equation}
  the measure $\Gamma^*_{2,K}[f]$ is nonnegative, satisfies
   \begin{equation}
  \label{eq:81}
  \Gamma_{2,K}^*[f](X)
  \le \int_X \Big(\big(\Delta f\big)^2-K\Gq f\Big)\,\d\mm
  \end{equation} 
  and provides a representation of the
  $\GGamma_2$ multilinear
  form as follows:
  \begin{equation}
    \label{eq:52}
    \GGamma_2(f;\varphi)=\int_X
    \tilde\varphi\,\d\,\Emeas Kf +K\int_X \Gq f\varphi\,\d\mm,\quad
    \forevery \varphi\in \dom_{L^\infty}(\Delta).
  \end{equation}
  \item 
    There exists a continuous, symmetric and bilinear map 
     $\gamma_{2,K}:\big(\dom_\V(\Delta)\cap \GLip\big)^2\to L^1(X,\mm)$ 
     such that for every $f\in \dom_\V(\Delta)\cap \GLip$
      (so that $\Gq f\in \Vinfty$) 
      there holds 
     \begin{equation}
       \label{eq:2}
       \Gamma_{2,K}^\star[f]=\ebilmeas K ff\mm+\Gamma_{2}^\perp[f],
       \quad\text{with}\quad \Gamma_{2}^\perp[f]\geq 0,\quad \Gamma_{2}^\perp[f]\perp \mm.
     \end{equation}
         Setting $\emeas Kf:=\ebilmeas Kff\ge 0$, one has 
     for every $f\in \DG$
  \begin{align}
     \label{eq:57}
  \Gq{\Gq f}&\le 4\emeas K f\,\Gq f\quad
  \text{$\mm$-a.e.~in $X$.}
\end{align}
\end{enumerate}
\end{theorem}

Notice that the measures $\Gamma_{2,K}^\star [f]$, $K\in \R$, just
differ by a multiple of $\Gq f\mm$, so the (nonnegative) singular part in 
the Lebesgue decomposition \eqref{eq:2} is independent of $K$.
In all the relevant estimates, it would be sufficient to consider 
the Lebesgue density $\gamma_{2,K}[\cdot]$, but it is still useful
to think in terms of measures to recover all the information 
coded inside $\GGamma_2(\cdot;\cdot)$.

If $\BE(K,N)$ holds and $\cE$ is quasi-regular, we have
the refined inequalities
\begin{equation}
  \label{eq:56}
  \Gamma_{2,K}^\star[f]\ge \emeas Kf\mm\ge \nu(\Delta f)²\mm.
\end{equation}

\subsection{Measure-valued $\Gamma_2$ tensor under 
lower regularity assumptions}
In this section we want to show that 
the regularity assumptions in Theorem~\ref{thm:mainquote} can be
considerably relaxed: 
in particular we will give a meaning to $\Gamma_{2,K}^\star[f]$ 
for every $f\in \dom_{L^4}(\Delta).$ The main tools are the a-priori
estimates
of Theorem~\ref{thm:crucial}
and the following simple approximation result.
\begin{lemma}
  \label{lem:Approximation2}
  Let us assume $\BE(K,\infty)$ and $p\in (1,\infty)$. 
  For every $f\in \dom_{L^p}(\Delta)$ there exist
  $f_n\in \dom_{\V}(\Delta)\cap \GLip$ converging to 
  $f$ in $\dom_{L^p}(\Delta)$.
\end{lemma}
\begin{proof}
  We argue as in the proof of Lemma~\ref{lem:Approximation}: by setting $h=f-\Delta f\in L²\cap
  L^p(X,\mm)$ and considering the truncated functions $h_n\in L²\cap
  L^\infty(X,\mm)$ as in \eqref{eq:6} and $f_n$ given by
  \eqref{eq:fn}, it is immediate to see that $f_n$ converge to $f$ in
  $\dom_{L^p}(\Delta)$ thanks to \eqref{eq:39}. 
  On the other hand, $f_n\in\dom_{L^\infty}(\Delta)$, so that $f_{n,\eps}:=\mathfrak H_{\eps}f_n$
  (where $\mathfrak H_\eps$ is given by \eqref{eq:55}) belong to $\DG$
  and converge to $f_n$ in $\dom_{L^p}(\Delta)$ as $\eps\down0$.
  A simple diagonal argument exhibits a sequence $f_{n,\eps_n}$ 
  satisfying the thesis of the Lemma.
\end{proof}
\begin{theorem}
  \label{thm:crucial}
  Let us assume that $\BE(K,\infty)$ holds and let $f,\,g\in
  \dom_{L^4}(\Delta)$.
  Then $\Gbil fg\in \V$ and 
  for every $\lambda\ge (K-1/2)_-$ and $f,\,g\in \dom_{L^4}(\Delta)$ we
  have (with non-optimal constants)
  \begin{equation}
    \label{eq:10}
    \|\Gq f\|_\V\le 
    2\sqrt {10}
    \,  \|\Delta_\lambda f\|_{L^4}^2,\quad
    \|\Gbil fg\|_\V\le 4\sqrt {10} \, \|\Delta_\lambda f\|_{L^4}\,\|\Delta_\lambda g\|_{L^4}.
  \end{equation}
  In particular, if $\BE(K,N)$ holds, for every $f\in \dom_{L^4}(\Delta)$ and $\varphi\in
  \V_+$ we have
  \begin{equation}
    \label{eq:11}
    \int_X \Big(-\frac 12\Gbil{\Gq f}\varphi+
      \Delta f\,\Gbil f\varphi+\varphi(\Delta f)^2\Big)\,\d\mm\ge 
      \int_X \Big(K\,\Gq f+\nu(\Delta f)^2\Big)\varphi\,\d\mm
  \end{equation}
  and for every $\varphi\in \V_+$  with $\Gq \varphi\in L^\infty(X,\mm)$,
  setting $K_\lambda:=2K+2\lambda\ge 1$, there holds
  \begin{equation}
    \label{eq:12}
    \begin{aligned}
      \int_X \Big(\Gq{\Gq f}+K_\lambda \Gq f^2\Big)\varphi\,\d\mm 
      &\le 2\int_X \Big(\Delta_\lambda f\Gbil{f}{\Gq f}+
      \Delta f\,\Delta_\lambda f\Gq       f
      \Big)\varphi\,\d\mm
      \\&+
      \int_X \Big(-\Gq f\Gbil{\Gq f}\varphi
      +
      2\Delta_\lambda f\,\Gq f\Gbil{f}\varphi
      \Big)\,\d\mm.
    \end{aligned}
  \end{equation}
\end{theorem}
\begin{proof} For every $f\in \DG$, \eqref{eq:G3} and an integration
by parts immediately yield
\begin{align}
  \notag
  K_\lambda\int_X \Gq f^2\,\d\mm+
  \cE(\Gq f)&\le -2\int_X \Gq f\Gbil{f}{\Delta_\lambda f}\,\d\mm
  \\&=
  2\int_X \Big(\Delta_\lambda f\,\Delta f \, \Gq f+\Delta_\lambda
  f\,\Gbil{f}{\Gq f}\Big)\,\d\mm.
  \label{eq:7}
\end{align}
By \eqref{eq:39} and the H\"older, inequality the right hand side of \eqref{eq:7} can be bounded by
  \begin{align*}
    4\|\Delta_\lambda f\|^2_{L^4}&\|\Gq f\|_{L^2}^{\phantom 2}+
    2\|\Delta_\lambda f\|_{L^4}^{\phantom {1/2}}\,\|\Gq f\|_{L^2}^{1/2}\|\Gq{\Gq
      f}\|_{L^1}^{1/2}
    \\&\le 
    \frac 14 \|\Gq f\|^2_{L^2}+
    16\|\Delta_\lambda f\|_{L^4}^4+
    \frac 12\cE(\Gq f)+\frac 14 \|\Gq f\|^2_{L^2}+4
    \|\Delta_\lambda f\|_{L^4}^4
  \end{align*}
  which yields the first estimate of \eqref{eq:10}. 
  The second one can be obtained by polarization, see the next Remark~\ref{rem:polarization}. 
  
  Now we use Lemma~\ref{lem:Approximation2} to approximate 
  any $f\in \dom_{L^4}(\Delta)$ with a sequence $f_n$ in $\DG$ converging to 
  $f$ in $\dom_{L^4}(\Delta)$ and we pass to the limit in
  \eqref{eq:10}
  by using the obvious bounds
  \begin{equation}\label{eq:sanvito}
    \big\|\Gq {f_n}-\Gq{f_m}\big\|_\V=
    \big\|\Gbil{f_n-f_m}{f_n+f_m}\big\|_\V\le 
    4\sqrt{10}\big\|\Delta_\lambda (f_n-f_m)\big\|_{L^4}
    \big\|\Delta_\lambda (f_n+f_m)\big\|_{L^4}.
  \end{equation}
  By the regularity of $\Gq f$ we can easily integrate by parts
  \eqref{eq:Reading} obtaining 
    \begin{equation}\label{eq:Reading2}
      \GGamma_2(f;\varphi)=\int_X -\frac 12\Gbil{\Gq f}\varphi+
      \Delta f\,\Gbil f\varphi+\varphi(\Delta f)^2\,\d\mm
  \end{equation}
  and thus, if $\BE(K,N)$ holds, \eqref{eq:11}.
  If $\varphi\in \V_+$ is bounded with $\Gq \varphi\in L^\infty(X,\mm)$, the inequality \eqref{eq:12} 
  is an immediate consequence of \eqref{eq:11} with $\nu=0$, by replacing $\varphi$ with $\Gq f\varphi$. In the general case
  $\varphi\in \V_+$ with $\Gq \varphi\in L^\infty(X,\mm)$ we use a truncation argument.
\end{proof}
\begin{remark}
  \label{rem:polarization}
  \upshape
  If $A,\,B$ are normed spaces and
  $G:A\times A\to B$ is a symmetric bilinear map satisfying 
  $\|G(a,a)\|_B\le C\|a\|_A^2$ for every $a\in A$, then 
  $G$ is continuous and satisfies
  \begin{equation}
    \label{eq:16}
    \|G(a_0,a_1)\|_B\le 2C \,\|a_0\|_A\|a_1\|_A\quad
    \forevery a_0,a_1\in A.
  \end{equation}
  It is sufficient to apply the polarization identity to $G$
  to obtain the estimate
  \begin{displaymath}
    \|G(a_0,a_1)\|_B\le \frac
    C4\Big(\|a_0+a_1\|_A^2+\|a_0-a_1\|_A^2\Big)
    \le C\Big(\|a_0\|_A^2+\|a_1\|_A^2\Big).
  \end{displaymath}
  Then, substituting $a_0$ by $\lambda a_0$ and $a_1$ by
  $\lambda^{-1}a_1$ and optimizing w.r.t.~the parameter $\lambda>0$
  the inequality \eqref{eq:16} follows.
\end{remark}
\begin{corollary}
  \label{cor:crucial}
  Assume that $\BE(K,\infty)$ holds. Then, for every $f\in \dom_{L^4}(\Delta)$ the linear functional 
  \begin{equation}
    \label{eq:61}
    \V\ni\varphi \mapsto 
    \int_X -\frac 12\Gbil{\Gq f}\varphi+
    \Delta f\,\Gbil f\varphi+\big((\Delta f)^2-K\Gq f\big)\varphi\,\d\mm
  \end{equation}
  belongs to $\cVl +'$ and can be represented by a measure that satisfies
  \eqref{eq:81} and \eqref{eq:57}, with $\emeas Kf$ defined as in \eqref{eq:2}.
  We still denote this measure by $\Gamma_{2,K}^\star[f]$.
\end{corollary}
\begin{proof}
  Let us denote by $\ell_f\in \V'$ the functional in \eqref{eq:61}. By
  \eqref{eq:11} it is immediate to see that $\ell_f\in \cVl+'$ so that
  we can apply the representation result stated in Proposition \ref{prop:Daniell};
  moreover, the first inequality in \eqref{eq:10} gives
  \begin{eqnarray}
    \label{eq:3}
    \Big|\langle\ell_f,\varphi\rangle\Big|&\le& 
    \|\varphi\|_{\V}\Big(\frac 12\|\Gq f\|_\V+
    \|\Delta f\|_{L^4}\|\Gq f\|_{L²}^{1/2}+\|\Delta f\|^2_{L^4}+|K|\|\Gq f\|_{L^2}\Big)
    \\&\le& C_K \|\varphi\|_{\V}\,\|f-\Delta f\|^2_{L^4}\nonumber
  \end{eqnarray}
  for all $\varphi\in\V$.
  
  In order to prove \eqref{eq:57} and \eqref{eq:81} we apply Lemma~\ref{lem:Approximation2}, 
  obtaining $f_n\in \DG$ converging to $f$ in 
  $\dom_{L^4}(\Delta)$. We first observe that 
  Theorem~\ref{thm:mainquote}, 
  the convergence of 
  $\Delta f_n$ in $L^4(X,\mm)$ and the convergence of
  $\Gq{f_n}$ in $\V$ coming from \eqref{eq:sanvito}, together with
  \eqref{eq:3}, yield
  \begin{equation}
    \label{eq:5}
    \lim_{n\to\infty}\int_X \varphi\,\d\Gamma_{2,K}^*[f_n]  
    =\int_X \varphi\,\d\Gamma_{2,K}^*[f]\quad\forevery \varphi\in \V.
  \end{equation}
  Passing to the limit as $n\to\infty$
  in the inequality (derived from the fact that $f_n$ satisfy \eqref{eq:81})
  $$
  \int_X \varphi\,\d\Gamma_{2,K}^*[f_n]\leq    \int_X \Big(\big(\Delta f_n\big)^2-K\Gq {f_n}\Big)\,\d\mm
  $$
  we obtain the same inequality with $f$ in place of $f_n$; then, \eqref{eq:3} and Proposition~\ref{prop:Daniell}(3) provide \eqref{eq:81} for $f$.
  In addition, we can still use the strong convergence of $\Gq {f_n}$ to $\Gq f$ in $\V$
  to show that
  $\big(\Gq{f_n}\big)^{1/2}$ and 
  $\big(\Gq{\Gq {f_n}}\big)^{1/2}$ converge to 
  $\big(\Gq f\big)^{1/2}$ and to $\big(\Gq{\Gq
    {f}}\big)^{1/2}$ in $L^2(X,\mm)$
  respectively. 
  Since the functions
  $g_n:=(\emeas K{f_n})^{1/2}$ are uniformly bounded in $L^2(X,\mm)$ 
  thanks to \eqref{eq:81}, up to extracting a weakly converging
  subsequence,
  it is not restrictive to assume that $g_n\weakto g$ in $L^2(X,\mm)$
  as $n\to\infty$ so that
  for every essentially bounded $\varphi\in \cVl+$ 
  \begin{equation}
    \label{eq:58}
    \begin{aligned}
      \int_X \big(\Gq{\Gq {f}}\big)^{1/2}\varphi\,\d\mm
      &=\lim_{n\to\infty}\int_X \big(\Gq{\Gq
        {f_n}}\big)^{1/2}\varphi\,\d\mm
      \\&\le 2\lim_{n\to\infty}\int_X g_n
      \Gq{f_n}^{1/2} \varphi\,\d\mm= 2\int_X g\Gq f^{1/2}\varphi\,\d\mm.
    \end{aligned}
  \end{equation}
  On the other hand, for every essentially bounded $\psi\in \cVl+'$ we
  obtain from \eqref{eq:5}
  \begin{equation}
    \label{eq:59}
    \begin{aligned}
      \int_X g^2\psi\,\d\mm&\le \liminf_{n\to\infty}
      \int_X g_n^2\psi\,\d\mm\le 
      \lim_{n\to\infty}
      \int_X \psi\,\d\Gamma_{2,K}^\star[f_n]=
      \int_X \psi\,\d\Gamma_{2,K}^\star[f]
      \\&=\int_X \psi \emeas kf\,\d\mm 
      +\int_X \psi\,\d\Gamma_{2,K}^\perp[f],
    \end{aligned}
  \end{equation}
  so that $g^2\le \emeas kf$ $\mm$-a.e.~in $X$.
  Combining with \eqref{eq:58} and taking the squares we eventually
  get 
  \eqref{eq:57}.
\end{proof}

\section{Metric measure spaces and their localization}\label{sec:LTGComp}

\subsection{Metric measure spaces, weak gradients and Cheeger energy}
\label{subsec:MMS}
We refer to the papers \cite{AGS11a}, \cite{AGS11b}, \cite{AGS12} for the basic facts and terminology on calculus in metric 
measure spaces; we will use the notation $W^{1,2}(X,\sfd,\mm)$ for the
Sobolev space, $\C$ for the Cheeger energy arising from the relaxation
in $L^2(X,\mm)$ 
of the local Lipschitz constant 
\begin{equation}
  \label{eq:20}
  |\rmD f|(x):=\limsup_{y\to x}\frac{|f(y)-f(x)|}{\sfd(y,x)},\quad
  f:X\to \R,
\end{equation}
of Lipschitz maps, $\weakgrad{f}$ for the so-called minimal relaxed gradient.

From now on, we shall denote by $\XX$ the class of metric measure spaces $(X,\sfd,\mm)$ satisfying the following three conditions:
\begin{itemize}
\item[(a)] $(X,\sfd)$ is complete and separable;
\item[(b)] $\mm$ is a nonnegative Borel measure with $\supp(\mm)=X$, satisfying
\begin{equation}\label{eq:grygorian}
\mm(B_r(x))\leq c\,\rme^{Ar^2}
\end{equation}
for suitable constants $c\geq 0$, $A\geq 0$. 
\item[(c)] $(X,\sfd,\mm)$ is infinitesimally Hilbertian according to the terminology introduced in \cite{Gigli12}, i.e., the Cheeger
energy $\C$ is a quadratic form. 
\end{itemize}
As explained in \cite{AGS11b}, \cite{AGS12}, the quadratic form $\C$ canonically 
induces a strongly regular Dirichlet $\cE$
form in $(X,\tau)$, where $\tau$ is the topology induced by $\sfd$. In addition, but this fact is less elementary (see \cite[\S4.3]{AGS11b}), the formula
$$
\Gq f=\weakgrad f^2,\quad
\Gbil{f}{g}=\lim_{\epsilon\downarrow 0}\frac{\weakgrad{(f+\epsilon g)}^2-\weakgrad{f}^2}{2\epsilon}
\quad\qquad f,\,g\in W^{1,2}(X,\sfd,\mm) $$%
(where the limit takes place in $L^1(X,\mm)$) provides an explicit expression  
of the \emph{Carr\'e du Champ} $\Gamma:\dom(\cE)\times \dom(\cE)\to
L^1(X,\mm)$ and 
yields the pointwise
upper estimate
\begin{equation}
  \label{eq:19}
  \Gq f\le |\rmD f|^2\quad\text{$\mm$-a.e.\ in $X$,
    whenever }f\in \Lip(X)\cap L^2(X,\mm),\quad |\rmD f|\in L^2(X,\mm).
\end{equation}
Eventually, \eqref{eq:grygorian} ensures that the 
generated Markov semigroup $(\sfP_t)_{t\ge0}$ is mass-preserving, 
so that \eqref{eq:18} and the formalism of 
Section~\ref{sec:BE} applies to the class of metric measure spaces in
$\XX$; in particular we can identify $W^{1,2}(X,\sfd,\mm)$ with
$\V$. 

The above discussions justify the following natural definition
(equivalent to the
  $\RCD^*(K,N)$ condition, see the next Section~\ref{sec:RCD}).

\begin{definition}[Metric $\BE(K,N)$ condition for metric measure
  spaces]
  \label{def:topo-BE}
  We say that \\
  $(X,\sfd,\mm)\in \XX$ satisfies the
  \emph{metric $\BE(K,N)$ condition}
  if 
  the Dirichlet form associated to the Cheeger
  energy of $(X,\sfd,\mm)$ satisfies $\BE(K,N)$ according to
  Definition~\ref{def:BE} and
  any
  \begin{equation}
    \label{eq:26}
    \text{$f\in W^{1,2}(X,\sfd,\mm)\cap L^\infty(X,\mm)$ with
      $\big\|\Gq f\big\|_{L^\infty}\le 1$ has a $1$-Lipschitz representative.}
  \end{equation}
\end{definition}
It is worth noticing that if $(X,\sfd,\mm)$ satisfies the metric
$\BE(K,\infty)$ condition then $\sfd$ coincides with the 
intrinsic distance $\sfd_\cE$ induced by $\cE$ and 
$(X,\sfd)$ is a length space (recall that $(X,\sfd)$ is a length space if
the distance between
two arbitrary points in $X$ is the infimum of the length of the absolutely continuous
curves connecting them). More precisely, the inequality $\sfd_{\cE}\leq \sfd$ is a direct consequence
of \eqref{eq:26}, while the curvature condition is involved in the proof of the converse inequality.

In this section we see how these concepts can be localized, building suitable cutoff functions with good second order regularity 
properties 
and a partition of unity subordinated to an open covering. As an
application, we see how the metric 
$\BE(K,N)$ condition
can, to some extent, be globalized (at least in locally 
compact metric spaces). 

The results of this section could be put in a more
abstract setup, as we did in \S\ref{sec:BE}, assuming the existence of cutoff functions $f$ with $\Gamma(f)\in L^\infty(X,\mm)$ separating sets with
positive distance. 
However, since the results we aim to are relative to metric measure spaces, we prefer to state
them in this setting, 
where, as a simple but useful application of \eqref{eq:19}, we can easily construct
cutoff functions with bounded weak gradient.
To this aim, we consider the distance-functions and 
the corresponding neighbourhoods
\begin{equation}
  \label{eq:31}
  \sfd(x,F):=\inf_{y\in F}\sfd(x,y),\quad
  F^{[h]}:=\big\{x\in X:\sfd(x,F)\le h\big\},\quad \text{for }F\subset
  X,\ x\in X,\ h\ge0,
\end{equation}
and we compose them with a function $\eta\in \Lip_{\rmc}(\R)$ with bounded support, 
so that $\nchi:=\eta\circ \sfd(\cdot,F)$ has bounded support;
it is immediate to see that 
\begin{equation}
  \label{eq:21}
  \nchi=\eta\circ\sfd(\cdot,F)\text{ belongs to $\V$, with}\quad \Gq\nchi
  \le \big|\eta'\big(\sfd(\cdot,F)\big)\big|^2\quad\mm\text{-a.e.~in }X.
\end{equation}
Let us conclude this introductory part
with two simple remarks concerning
proper metric spaces and
and the regularity of the Cheeger
energy.
Recall that a metric space $(X,\sfd)$ is called \emph{proper} if
every closed bounded subset is compact.
\begin{remark}[Proper metric spaces and the length condition]
  \label{rem:proper}
  \upshape
  Every complete and 
  locally compact metric space $(X,\sfd)$ is also
  proper if it satisfies a length condition 
  (see e.g.\ \cite[Prop.~2.5.22]{Burago-Burago-Ivanov01}).
  Properness immediately yields that this
  infimum
  is also attained so that a locally compact length space is proper
  and geodesic.  
  This characterization is well adapted to our situation, since
  every m.m.s.~$(X,\sfd,\mm)$ satisfying the metric
  $\BE(K,N)$ condition is a length space and it is also
  locally compact (thus proper and geodesic) if $N<\infty$.    
\end{remark}
\begin{remark}[Regularity of the Cheeger energy in proper 
  metric spaces]
  \label{rem:Cheeger-regular}
  \upshape
  Recalling Remark \ref{rem:regular}, it is immediate to check
  that the Cheeger energy is a regular (thus a fortiori quasi-regular)
  Dirichlet form whenever $(X,\sfd)$ is proper.
  In fact, it is easy to see that every function with finite
  energy can be approximated in $W^{1,2}(X,\sfd,\mm)$ by functions
  with bounded support (see e.g.~\cite[Lemma 4.11]{AGS11a}) 
  and those functions are limits in $W^{1,2}(X,\sfd,\mm)$ of sequences of 
  Lipschitz functions with bounded support.
  The same approximation property holds for uniform convergence
  and any function in $\rmC_\rmc(X)$ by Arzel\`a-Ascoli Theorem.
\end{remark}

\subsection{Localization of metric measure spaces}
\label{subsec:localization}
In connection with localization-globalization of spaces $(X,\sfd,\mm)\in\XX$, the following properties of the relaxed gradient will be useful
(see \cite[Theorem~4.19]{AGS11b} for the proof).

\begin{proposition}[Localization of relaxed gradients and $\XX$]\label{prop:localK}
For $U\subset X$ open, let us consider the metric measure space $(\bar{U},\sfd,\mm \llcorner \bar{U})$ and let us denote by
$|{\rm D} f|_{w,\bar{U}}$ the minimal relaxed gradient in the new space. Then:
\begin{itemize}
\item[(a)] $f\in W^{1,2}(X,\sfd,\mm)$ implies $f\in
  W^{1,2}(\bar{U},\sfd,\mm\llcorner\bar U)$ and $\weakgrad{f}=|{\rm D} f|_{w,\bar{U}}$ 
$\mm$-a.e. in $U$. Conversely, if $f\in W^{1,2}(\bar{U},\sfd,\mm\llcorner\bar U)$ and $\supp(f)$ has positive distance from $X\setminus U$, 
then $f$ extended with the 0 value to the whole of $X$ belongs to $f\in W^{1,2}(X,\sfd,\mm)$.
\item[(b)] If $\mm(\partial U)=0$ then $(\bar{U},\sfd,\mm\llcorner\bar U)\in\XX$. 
\end{itemize}
\end{proposition}

For $U\subset X$ open, $W^{1,2}_{\rm c}(U,\sfd,\mm)$ will denote the 
subspace of $W^{1,2}(X,\sfd,\mm)$ whose functions have compact support
in $U$. We will similarly consider $\Lip_\rmc(U)$.
We will occasionally identify a measurable function $f:U\to \R$ 
with compact support in $U$ with
its trivial extension $\tilde f$ to $X$ and viceversa. 

We can also introduce the localized versions of Lebesgue and Sobolev spaces on
open subsets of 
$X$. Even if not explicitly assumed, these notions are interesting
when $X$ is locally compact.
\begin{definition}[Local $L^p$ and Sobolev spaces]
  \label{def:local-Sobolev}
  Let $U\subset X$ be open and non-empty and
  let $f:U\to \R$ be a $\mm$-measurable map.
  We say that $f\in L^p_{\rm  loc}(U,\mm)$, $p\in [1,\infty]$, if $f\restr{E}\in L^p(E,\mm\llcorner E)$ for every
  compact subset $E\subset U$.
  We say that $f\in W^{1,2}_{\rm loc}(U,\sfd,\mm)$ if 
  for every compact set $E\subset U$ there exists $f_E\in
  W^{1,2}(X,\sfd,\mm)$ such that $f=f_E$ $\mm$-a.e.~in $E$.
  For every $f,g\in W^{1,2}_{\rm loc}(U,\sfd,\mm)$ we can then define
  $\Gbil fg\in L^1_{\rm loc}(U,\mm)$ by $\Gbil fg\restr E:=\Gbil {f_E}{g_E}$.  
\end{definition}
It is not difficult to check that the above definition is consistent 
thanks to the locality property of $\Gamma$. 
We can also easily check the equivalent characterization in terms
of cutoff function (used for instance in \cite{Gigli-Mondino12}): 
\begin{equation}
f\in W^{1,2}_{\rm loc}(U,\sfd,\mm)\quad\text{iff}\quad
\widetilde{f\nchi}\in W^{1,2}(X,\sfd,\mm)\quad\forevery
\nchi\in \Lip_\rmc(U).\label{eq:62}
\end{equation}
Notice that
if $f\in W^{1,2}_{\rm
    loc}(U,\sfd,\mm)$ and $h\in W^{1,2}_{\rm
    c}(U,\sfd,\mm)$ we have $\widetilde \Gamma(f,h)\in L^1(X,\mm)$ with
  \begin{equation}
    \label{eq:35}
    \widetilde \Gamma(f,h)=\Gbil{\widetilde{f\nchi}}h\quad\text{whenever }\nchi\in \Lip_\rmc(U),\
    \nchi\equiv 1\text{ on }\supp(h).
  \end{equation}

Let us now consider the localization property of the Laplace operator.
\begin{lemma}[Global to local for the Laplacian of compactly
  supported functions]
  \label{rem:locality}
  \ \\Let 
$(X,\sfd,\mm)\in\XX$ and, given an open set $U\subset X$, assume that  $f\in W^{1,2}(X,\sfd,\mm)$ has support with positive 
distance from $X\setminus U$. Then $\Delta f\in L^2(X,\mm)$ if and only if $\Delta_{\bar{U}} f\in L^2(\bar{U},\mm\llcorner {\bar{U}})$.
In addition $\Delta_{\bar{U}} f = \Delta f$ $\mm$-a.e. in $\bar{U}$
and $\supp(\Delta f)\subset\supp(f)$.
\end{lemma}
\begin{proof}
Let us assume that $\Delta f\in L^2(X,\mm)$. First of all, choosing
Lipschitz functions $g$  in \eqref{eq:defDelta} with compact support in $G:=X\setminus\supp(f)$ 
(these functions are dense in $L^2(G,\mm\llcorner G)$ by a simple
truncation argument)  we see that
$\Delta f=0$ $\mm$-a.e. in $X\setminus\supp(f)$, i.e. $\supp(\Delta f)\subset\supp(f)$.
Now, for every $\psi \in W^{1,2}(\bar{U},\sfd,\mm\llcorner{\bar{U}})$ we can apply Proposition~\ref{prop:localK} and a multiplication by
a cutoff function as in \eqref{eq:21} with $F=\supp(f)$  
to find another function $\tilde\psi\in W^{1,2}(X,\sfd,\mm)$ coinciding with $\psi$ in a neighbourhood of $\supp(f)$;
thanks to the locality of $\Gamma$ we have then
$$\int_{\bar{U}} \Gamma(\psi,f) \, \d \mm= \int_X \Gamma(\tilde\psi,f)\, \d \mm = - \int_X \tilde\psi\,\Delta f  \, \d \mm =  -\int_{\bar{U}} \psi\,\Delta f \, \d \mm. $$
The proof of the converse implication is similar.
\end{proof}

\begin{lemma}[Construction of smoother cutoff functions]\label{lem:cutoff}
Let $(X,\sfd,\mm)\in\XX$ and let $U\subset X$ be an open subset such
that $(\bar{U},\sfd,\mm\llcorner{\bar{U}}) \in \XX$ satisfies the metric $\BE(K,\infty)$ condition. 

Then, for all $E\subset U$ compact and $G\subset X$ open
and relatively compact with $E\subset G$ and $\bar{G}\subset U$
there exists a Lipschitz function $\hat\nchi:X\to \R$ satisfying:
\begin{itemize}
\item[(i)] $0\le \hat\nchi\le 1$,
  $\hat\nchi\equiv 1$ on a 
  neighbourhood $E^{[h]}$ of $E$ for some $h>0$, and $\supp(\hat\nchi) \subset G$;
\item[(ii)] $\Delta\hat\nchi \in L^\infty(X,\mm)$ and 
  $\Gq{\hat\nchi}\in W^{1,2}(X,\sfd,\mm)$.
\end{itemize}
\end{lemma}
\begin{proof}
Notice that the point is to construct a cut off function with
$L^\infty$ Laplacian, indeed  the existence of a Lipschitz cut off
function $f:X\to [0,1]$  
satisfying (i)
is trivial by \eqref{eq:21}, 
since $\eps:=\inf_{x\not\in G}\sfd(x,E)>0$;
we can thus suppose that 
$f\equiv 1$ on $E^{[\eps/3]}$ and  $\supp(f)\subset E^{[2\eps/3]}$.

At first we regularize $f$ via the mollified heat flow  ${\mathfrak
  H}_U^t$ of the m.m.s. $(\bar{U},\sfd,\mm\llcorner{\bar{U}})$ defined
as in \eqref{eq:55} 
by setting
$$f_t:= {\mathfrak H}_U^t f.$$
Since by assumption $(\bar{U},\sfd,\mm\llcorner{\bar{U}})$ satisfies 
the metric $\BE(K,\infty)$ condition,
by the pointwise gradient estimate \eqref{eq:77} together with the
maximum principle and \eqref{eq:26}, we know that $\{f_t\}_{t\in [0,\delta]}$ are uniformly Lipschitz
on $\bar{U}$ for any $\delta>0$, and moreover $\Delta f_t \in
L^\infty(\bar{U},\mm\llcorner{\bar{U}})$. Combining  Arzel\`a-Ascoli
theorem and the fact that $f_t\to f$ in $L^2(\bar{U},\mm)$ as $t\to 0$
it follows that $f_t\to f$ uniformly on $\bar{G}$. Therefore,
recalling also the maximum principle for the heat flow, for $\delta>0$
small enough we have
\begin{equation}
\frac{3}{4}\leq f_\delta \leq 1 \text{ on } E^{[\eps/3]} \quad  \text{ and }
\quad 0 \leq f_\delta \leq \frac{1}{4} \text{ on } \bar{G}\setminus E^{[2\eps/3]}.
\end{equation} 
Now  let $\eta \in C^2([0,1],[0,1])$ be such that $\eta([0,1/4])=\{0\}$ and $\eta([3/4,1])=\{1\}$. It is immediate to check,
using Lemma~\ref{rem:locality}, Proposition~\ref{prop:localK} 
and Theorem \ref{thm:crucial} (applied to the 
m.m.s.~$(\bar U,\sfd,\mm\llcorner\bar U)$) 
that the trivial extension of $\hat\nchi:=\eta\circ f_\delta$ outside
$\bar G$ has the desired properties with $h:=\eps/3$.
\end{proof}

\begin{remark}
  \label{rem:loc-comp}
  \upshape
  Notice that whenever $E\subset U$, with $E$ compact and $U$ open and locally compact, 
  we can always find an open and relatively compact
  neighbourhood $G$ of $E$ as in the above Lemma: since for every
  $x\in E$ there exists a relatively compact open ball $B_x$ with
  $\overline{B_x}\subset U$, it is sufficient to set $G:=\cup_{x\in
    E_0} B_x$ where $E_0\subset E$ is a finite set such that
  $(B_{x})_{x\in E_0}$ is an open cover of $E$.
\end{remark}
The lemma above easily provides the following proposition, stating the existence of a regular partition of unity.
\begin{proposition}[Partition of unity]\label{pro:PartOf1}
Let $(X,\sfd,\mm)\in\XX$, $E\subset X$ compact and let  
$U=\cup_{i\in I} U_i$ where $\{U_i\}_{i\in I}$ is a covering of $E$ by non-empty locally compact 
open subsets such that the m.m.s. $(\bar{U}_i,\sfd,\mm\llcorner{\bar{U}_i})\in \XX$  satisfy the
metric 
$\BE(K,\infty)$ condition. Then there exist Lipschitz functions $\nchi_i:X \to[0,\infty)$, null for all but finitely many $i$ and
satisfying:
\begin{itemize}
\item[(i)] $\supp(\nchi_i) \subset U_i$ is compact 
 and $\sum_i \nchi_i \equiv 1$ on a neighbourhood of $E$; 
\item[(ii)] $\Delta \nchi_i \in L^\infty(X,\mm)$ and  
  $\Gq {\nchi_i}\in W^{1,2}(X,\sfd,\mm)$. 
\end{itemize} 

In particular $\psi:=\sum_i \nchi_i$ satisfies $\psi\equiv 1$ on $E$ and
\begin{equation}
  \label{eq:32}
  \psi\in \Lip_\rmc(U),\quad
  0\le \psi\le 1,\quad
  \Delta \psi\in L^\infty(X,\mm),\quad
  \Gq{\psi}\in W^{1,2}(X,\sfd,\mm).
\end{equation}
\end{proposition}
\begin{proof} By the compactness of $ E$ we can assume with no
  loss of generality that $I$ is finite. Since the continuous function
  \begin{displaymath}
    E\ni x\mapsto 
    \max_{i\in I}\sfd(x,X\setminus U_i)
  \end{displaymath}
  has a minimum $a>0$ in $E$, 
  considering the sets
  $$
  E_i:=\left\{x\in E:\ {\rm dist}(x,X\setminus U_i)\geq 
    a/2\right\},\qquad
  G_i:=\left\{x\in X:\ {\rm dist}(x,E_i)< b\right\}
$$
for $b>0$ sufficiently small, we provide compact sets
$E_i\subset G_i$ and open 
relatively compact (by Remark~\ref{rem:loc-comp}) sets $G_i\subset X$ with
$E_i \subset G_i$ and $\bar{G}_i\subset U_i$, in such a way that
$\cup_{i\in I}E_i=E$ and 
the neighbourhood $E^{[h]}$ of $E$ is contained in $\cup_i G_i$ for
$h<b$. 

With this choice of $E_i$ and $G_i$, if we consider 
the cutoff functions $\hat{\nchi}_i$ provided by
Lemma~\ref{lem:cutoff}, we clearly have $\sum_i \hat{\nchi}_i \geq 1$
on $E^{[h]}$ for some positive $h<b$ sufficiently small. 
By Leibniz and chain rule, it is therefore clear that the functions
$$\nchi_i:=\frac{\hat{\nchi}_i}{\eta(\sum_i\hat{\nchi}_i)}=2\hat{\nchi}_i+\bigl(
\frac{1}{\eta(\sum_i\hat{\nchi}_i)}-2\bigr)\hat\nchi_i
$$
satisfy (i) and (ii) above, provided that we choose a smooth nondecreasing function $\eta(s)$ identically equal to $s$ on $[1,\infty)$ and
identically equal to $1/2$ on $[0,1/2]$.

In order to prove the regularity for $\Gq \psi$ of \eqref{eq:32} it is
sufficient to recall Proposition \ref{prop:localK} and 
\begin{equation}
  \label{eq:37}
  \Gq{\sum_i
    \nchi_i}=\sum_{i,\,j}\Gbil{\nchi_i}{\nchi_j},\qquad
  \Gbil{\nchi_i}{\nchi_j}=\frac 14\Gq{\nchi_i+\nchi_j}-
  \frac 14\Gq {\nchi_i-\nchi_j}.\qedhere
\end{equation}
\end{proof}
\begin{remark}
  \label{rem:trivial}
  \upshape
  Let $U=\cup_{i\in I}U_i$ as in the previous Proposition~\ref{pro:PartOf1}.
  Then a measurable function $f:U\to \R$ belongs to $W^{1,2}_{\rm
    loc}(U,\sfd,\mm)$ if and only if $\widetilde{f\psi}\in W^{1,2}(X,\sfd,\mm)$
  for every function $\psi$ as in \eqref{eq:32}. Thus we can 
  use better cutoff functions in \eqref{eq:62}.
\end{remark}

\noindent
We can now use the partitions of unity to prove a
local higher integrability and regularity of $\Gamma$. 

\begin{lemma}[Improved local integrability and regularity of $\Gamma(f)$]\label{lem:ImprInt}
Let $(X,\sfd,\mm)$ be a locally compact m.m.s.~in $\XX$ and 
let $X=\cup_{i\in I}U_i$ where 
$\{U_i\}_{i\in I}$ are non-empty open 
subsets such that 
 $(\bar{U}_i,\sfd,\mm\llcorner{\bar{U}_i})\in \XX$ satisfy the metric 
$\BE(K,\infty)$ condition for all $i\in I$. 

Then for every $f\in\dom_{L^4}(\Delta)\cap L^\infty(X,\mm)$ one has  
$\Gq{f\psi}\in W^{1,2}(X,\sfd,\mm)$
and $f\psi\in D_{L^4}(\Delta)$  for every 
cutoff function $\psi$ satisfying \eqref{eq:32}.
In particular, 
$\Gamma(f) \in W^{1,2}_{\rm loc}(X,\sfd,\mm)$.
\end{lemma}
\begin{proof} 
  Let us consider the compact set $E:=\supp(\psi)$ and 
  let $\{\nchi_i\}$, $I=\{1,\ldots,n\}$,
  be the partition of unity, subordinated to $E$ 
  and to the open covering $\{U_i\}$, constructed in
Proposition~\ref{pro:PartOf1};  
let $\hat\nchi_i$ be a cutoff function provided by Lemma~\ref{lem:cutoff} corresponding to the compact set $\supp(\nchi_i)$,
the open set $U_i$ and $G_i$ as 
in Remark~\ref{rem:loc-comp}.
We define
\begin{equation}
  \label{eq:34}
  \hat f_i:=\hat\nchi_i f,\quad
  f_i:=\nchi_i f,\quad \text{so that}\quad
  f_i=\nchi_i\,\hat f_i.
\end{equation}
Recalling Lemma~\ref{rem:locality}  and the Leibniz formula \eqref{eq:43},
it is easy to check that  $\hat f_i\in L^\infty\cap W^{1,2}(\bar{U}_i,\sfd,
\mm\llcorner{\bar{U}_i})$ and that $\Delta \hat f_i \in L^2(\bar{U}_i,
\mm\llcorner{\bar{U}_i})$. 
Since by assumption the m.m.s. $(\bar{U}_i,\sfd,
\mm\llcorner{\bar{U}_i})$ satisfies the $\BE(K,\infty)$ condition, by
the gradient interpolation Theorem \ref{thm:interpolation} with $p=2$
we obtain that $\Gamma(\hat f_i)\in L^2(\bar{U}_i,
\mm\llcorner{\bar{U}_i})$. 

Still applying \eqref{eq:43}, now with $f:=\hat f_i$ and
$\nchi:=\nchi_i$, we obtain that $f_i \in W^{1,2}(\bar
U_i,\sfd,\mm\llcorner \bar U_i)\cap L^\infty(\bar U_i,\mm)$ with
$\Delta_{\bar U_i} f_i\in L^4(\bar U_i,\mm\llcorner\bar U_i)$:
Theorem~\ref{thm:crucial} provides
$\Gq {f_i}\in W^{1,2}(\bar
U_i,\sfd,\mm\llcorner \bar U_i)$ and
since $\Gq {f_i}$ has compact support in $U_i$ we conclude that (the
trivial extension of) $\Gq{f_i}$ belongs to $W^{1,2}(X,\sfd,\mm)$ and that
$\Delta f_i\in L^4(X,\mm)$.
More generally, since $\nchi_j$ is globally Lipschitz and with bounded Laplacian in $X$,
we obtain that $f_i\nchi_j\in {\rm D}_{L^4}(\Delta)$ for all $i,\,j=1,\ldots,n$. Since both
$f_i\nchi_j$ and $f_j\nchi_i$ have compact support in $U_i$ we obtain 
$\Delta_{\bar U_i} (f_i\nchi_j\pm f_j\nchi_i)\in L^4(\bar U_i,\mm\llcorner\bar U_i)$.
By applying Theorem~\ref{thm:crucial} in $U_i$ we get
$\Gq {f_i\nchi_j\pm f_j\nchi_i}\in W^{1,2}(X,\sfd,\mm)$, so that
by polarization
$$
\Gbil{f_i}{f_j} =\Gbil{f_i\nchi_j}{f_j\nchi_i}\in W^{1,2}(X,\sfd,\mm).
$$
The bilinearity of $\Gamma$ yields
$$
\Gq{\sum_i f_i}=\Gamma\bigl(\sum_i f_i, \sum_j f_j\bigr)= 
\sum_{ij} \Gbil{f_i}{f_j} \in W^{1,2}(X,\sfd,\mm).$$

Using that $\sum_i \nchi_i\equiv 1$ on $E$ and the identity 
$f\psi=\sum_i (f\psi)\nchi_i=\psi\sum_i f_i$
we conclude that $\Gq{f\psi}\in W^{1,2}(X,\sfd,\mm)$ 
and $\Delta(f\psi)\in L^4(X,\mm)$.

Finally, $\Gq f\in W^{1,2}_{\rm loc}(X,\sfd,\mm)$ follows by Remark \ref{rem:trivial}.\end{proof}
\begin{theorem}
  \label{cor:local}
  Under the same assumptions of the previous Lemma~\ref{lem:ImprInt},
  every function $f\in \dom_{L^4}(\Delta)\cap L^\infty(X,\mm)$ satisfies
  \eqref{eq:11} for every nonnegative $\varphi\in W^{1,2}_{\rm c}(X,\sfd,\mm)$.

  If, moreover, $(X,\sfd)$ is a proper metric space 
  then there exists a nonnegative Radon measure $\Gamma_{2,K}^\star[f]$ 
  vanishing on $\cE$-polar sets and
  representing the linear functional
  \begin{equation}
    \label{eq:63}
    W^{1,2}_\rmc(X,\sfd,\mm)\ni\varphi \mapsto 
    \int_X -\frac 12\Gbil{\Gq f}\varphi+
    \Delta f\,\Gbil f\varphi+\big((\Delta f)^2-K\Gq f\big)\varphi\,\d\mm.
  \end{equation}
  Finally, the density $\emeas Kf$ of the measure $\Gamma_{2,K}^\star[f]$ still satisfies \eqref{eq:57}.
\end{theorem}
Notice that both \eqref{eq:11} and \eqref{eq:12} 
make sense 
under the above
assumptions 
thanks to 
Lemma~\ref{lem:ImprInt}.
\begin{proof}
  Let $f\in \dom_{L^4}(\Delta)\cap L^\infty(X,\mm)$ 
  and $\varphi\in W^{1,2}_\rmc(X,\sfd,\mm)$ be fixed and let us prove
  \eqref{eq:11}. 
  
Let $E=\supp(\varphi)$ and let $\{\nchi_i\}$ be the cutoff functions constructed in Proposition~\ref{pro:PartOf1}, subordinated to the open covering $\{U_i\}$,
whose sum is identically 1 in a neighbourhood of $E$, null for all but finitely many $i$;
since $\nchi_i$ have support contained in $U_i$, for all $i$ such that $\nchi_i$ is not null we apply Lemma~\ref{lem:cutoff} to 
obtain Lipschitz function $\hat\nchi_i$ with compact support in $U_i$,
bounded Laplacian, identically equal to 1 on a neighbourhood of
$\supp(\nchi_i)$.
We can also find $\psi\in \Lip_\rmc(U)$ satisfying \eqref{eq:32}
and $\psi\equiv 1$ on $\supp(\sum_i\hat\nchi_i).$

It is easy to check, using Lemma~\ref{lem:ImprInt} and 
Lemma~\ref{rem:locality}, that the functions $\varphi_i:=\nchi_i
\varphi$
 and $f_i:=\hat{\nchi}_i f$ satisfy the assumptions of 
Theorem~\ref{thm:crucial} in the metrically  $\BE(K,N)$ 
m.m.s.~$(\bar U_i,\sfd,\mm\llcorner\bar U_i)$. We thus get 
\begin{equation}\label{eq:LTGLoc}
\int_{\bar{U}_i} \Big( -\frac12 \Gbil{\Gq {f_i}}{\varphi_i}+  
\Delta f_i\,\Gamma(f_i,\varphi_i) + \varphi_i(\Delta f_i)^2\Big)\,\d\mm \geq K\int_{\bar{U}_i} \Gamma(f_i) \,\varphi_i\,\d\mm+\nu\int_{\bar{U}_i}
  (\Delta f_i)^2\varphi_i\,\d\mm.
\end{equation}

Since Lemma~\ref{lem:ImprInt} gives that $\Gamma(f)\in W^{1,2}_{\rm
  loc}(X,\sfd,\mm)$, 
recalling also that $\sum_i \varphi_i\equiv \varphi$ and $\hat{\nchi}_i\equiv 1$ on $\supp(\nchi_i)$,  we can write
\begin{align*}
  \int_X &\Big( -\frac 12\Gbil{\Gq f}\varphi +
      \Delta f\,\Gbil f\varphi+\varphi(\Delta f)^2\Big)\,\d\mm \\
      &= \int_X \left( -\frac 12\Gbil{\Gq f\,\psi}{\sum_i \varphi_i
          } +   \Delta f \,\Gamma\bigl(f,\sum_i \varphi_i \bigr) + 
 \bigl(\sum_i \varphi_i\bigr) (\Delta f)^2\right)\,\d\mm  \nonumber \\
 &=\sum_i \int_{\bar{U}_i} \left(-\frac12 \Gbil{\Gq f\,\psi}{\varphi_i} 
+   \Delta f \,\Gamma\left(f, \varphi_i \right) +  \varphi_i (\Delta f)^2\right)\,\d\mm  \nonumber \\
 &=\sum_i \int_{\bar{U}_i} \left(
-\frac12 \Gbil{\Gq {f_i}}{\varphi_i}  +   \Delta f_i \,\Gamma\left(f_i, \varphi_i \right) +  \varphi_i (\Delta f_i)^2\right)\,\d\mm  \nonumber \\
 &\geq \sum_i K\int_{\bar{U}_i}  \Gamma(f_i) \,\varphi_i\,\d\mm+\nu\int_{\bar{U}_i}  
  (\Delta f_i)^2\varphi_i\,\d\mm \quad \text{by \eqref{eq:LTGLoc}} \nonumber \\
&= 	\sum_i K\int_{\bar{U}_i}  \Gamma(f) \,\varphi_i\,\d\mm+\nu\int_{\bar{U}_i} 
  (\Delta f)^2\varphi_i\,\d\mm \nonumber \\
&=  K\int_X \Gamma(f) \,\varphi \,\d\mm+\nu\int_X
  (\Delta f)^2\varphi\,\d\mm.
\end{align*}

In order to prove the second part of the statement, 
in the case when $(X,\sfd)$ is proper  
(recall Remark~\ref{rem:proper}), let us call $\ell_f$ the linear functional
defined by \eqref{eq:63}, let us fix $x_0\in X$ with the
collection of the open balls $B_R:=\{x\in X:\sfd(x,x_0)<R\}$
and let us consider the Hilbert space
$\V_R$ obtained by taking the closure in $\V$ of
the set $\{\varphi\in W^{1,2}(X,\sfd,\mm):\supp\varphi\subset 
{B_R}\big\}$.
It is easy to check that the restriction of $\cE$ to $\V_R$ is
a regular Dirichlet form 
(recall Remark~\ref{rem:Cheeger-regular}) 
on $L^2(B_R,\mm\llcorner B_R)$ 
and the restriction of $\ell_f$ to $\V_R$ is a nonnegative continuous
functional, which also satisfies (3) of Proposition~\ref{prop:Daniell}:
in fact, if $\varphi\in\V_R$ with $0\le \varphi\le 1$ $\mm$-a.e., 
by taking a cutoff function $\psi$ as in \eqref{eq:32} with
$U\supset \overline B_R$ and $\psi\equiv 1$ on $\overline B_R$,
the inequality $0\le \psi-\varphi\in W^{1,2}_\rmc(X,\sfd,\mm)$ 
and \eqref{eq:11} yield
\begin{displaymath}
  \langle \ell_f,\varphi\rangle\le 
  \int_X -\frac 12\Gbil{\Gq f}\psi+
  \Delta f\,\Gbil f\psi+\big((\Delta f)^2-K\Gq f\big)\psi\,\d\mm.
\end{displaymath}
Thus, the action of 
$\ell_f$ on $\V_R$ can be represented
by a finite nonnegative Borel measure $\mu_R$ on $B_R$ 
not charging $\cE$-polar subsets of $B_R$.
It is easy to check that $S<R$ yields 
$(\mu_R)\llcorner B_S=\mu_S$, so that we can eventually find 
a nonnegative Radon measure $\Gamma_{2,K}^\star[f]$ as stated in the
Theorem.

Let us now take an arbitrary compact set $E\subset U_i$
and a Lipschitz function $\nchi_i$ with compact support in $U_i$,
bounded Laplacian and identically $1$ on a neighbourhood of
$\supp(\varphi_i)$. The function $\hat f_i=\nchi_i f$ belongs to 
$W^{1,2}(\bar U_i,\sfd,\mm\llcorner \bar U_i)$ with
$f_i,\Delta_{U_i}f_i\in L^4(\bar U_i,\mm\llcorner \bar U_i)$ 
so that applying Corollary \ref{cor:crucial} we
get a finite nonnegative Borel measure $\mu_i=\Gamma_{2,K}^\star[f_i]$ 
(relative to $U_i$) with Lebesgue density $\gamma_i$ 
satisfying
\begin{equation}
  \label{eq:64}
  \langle \ell_{f_i},\psi\rangle=\int_{\bar U_i}\psi\,\d\mu_i\quad
  \forevery\ \psi\in W^{1,2}_\rmc(U_i,\sfd,\mm),\quad
  \Gq{\Gq {f_i}}\le 4\gamma_i\Gq {f_i}\quad \text{$\mm$-a.e.\ on }U_i.
\end{equation}
For every function $\varphi\in
W^{1,2}(X,\sfd,\mm)$ with support in $E$ we then have
\begin{align*}
  \int_E \varphi\,\d\Gamma_{2,K}^\star [f]=
  \langle \ell_f,\varphi\rangle=
  \langle \ell_{f_i},\varphi\rangle=
  \int_{E}\varphi\,\d\mu_i
\end{align*}
so that $\Gamma_{2,K}^\star[f]$ coincide with $\mu_i$ on $E$. 
It follows that its Lebesgue density $\gamma$ coincides with
$\gamma_i$ and \eqref{eq:64} yields
\begin{displaymath}
  \Gq{\Gq f}=\Gq{\Gq{f_i}}\le 4\gamma_i\,\Gq {f_i}=
  4\gamma\,\Gq {f}\quad\text{$\mm$-a.e.~on }E.
\end{displaymath}
Since we can cover $X$ with a sequence of compact sets contained in some set $U_i$,
we conclude.
\end{proof}

\begin{theorem}
  \label{thm:medusa}
  Let $(X,\sfd,\mm)\in\XX$ be a proper 
  m.m.s. and let $X=\cup_{i\in I} U_i$ where
  $\{U_i\}_{i\in I}$ are non-empty open sets such that 
  $(\bar U_i,\sfd,\mm\llcorner\bar U_i)\in \XX$ 
  satisfy the metric $\BE(K,\infty)$ condition for all $i\in I$.
  For every $f\in \dom_{L^4}(\Delta)\cap L^\infty(X,\mm)$ the following properties hold:
  \begin{enumerate}[\rm (a)]
  \item the measure $\Gamma_{2,K}^\star[f]$ is finite and satisfies
    \eqref{eq:81};
  \item $\Gq f\in W^{1,2}(X,\sfd,\mm)$ and satisfies \eqref{eq:12} for any nonnegative $\varphi$ with
  $\Gq{\varphi}\in L^\infty(X,\mm)$.
  \end{enumerate}
\end{theorem}
\begin{proof}
  Let us fix $x_0\in X$ and 
  let us consider the Lipschitz cutoff function
  \begin{equation}\label{eq:65}
    \varphi_n(x)=
    \begin{cases}
      1&\text{on }B_{2^n}(x_0),\\
      0&\text{on }X\setminus B_{2^{n+1}}(x_0),\\
      2-\sfd(x,x_0)2^{-n}&\text{on }B_{2^{n+1}}(x_0)\setminus B_{2^{n}}(x_0),
    \end{cases}
  \end{equation}
  which satisfies $|\rmD\varphi_n|\le 2^{-n}$. By Theorem~\ref{cor:local},
  if $f\in \dom_{L^4}(\Delta)\cap L^\infty(X,\mm)$ and $\gamma$ is the Lebesgue density of
  $\Gamma_{2,K}^\star[f]$, 
  we get
  \begin{align*}
    -\frac 12 \int_X \Gbil{\Gq f}{\varphi_n^2}\,\d\mm&=
    -\int_X \Gbil{\Gq f}{\varphi_n}\,\varphi_n\,\d\mm\le 
    \int_X \varphi_n\sqrt{\Gq{\Gq f}}\sqrt {\Gq {\varphi_n}}\,\d\mm
    \\&\le 
    2\int_X \varphi_n\sqrt{\gamma{\Gq f}\Gq{\varphi_n}}
    \,\d\mm
    \le  \eps \int_X \gamma\varphi_n^2\,\d\mm+
    \frac 1{4^{n}\eps} \int_X \Gq f\,\d\mm,\\
    \int_X \Delta f\,\Gbil f{\varphi_n^2}\,\d\mm&\le 
    2^{-n}\|\Delta f\|_{L^2} \sqrt{\cE(f)}.
    \end{align*}
    In addition, the integrability of $\Gq{f}$ gives
    $$
    \limsup_{n\to\infty}\int_X \Big((\Delta f)^2-K\Gq f\Big)\varphi_n^2\,\d\mm\le 
    \int_X \Big((\Delta f)^2-K\Gq f\Big)\,\d\mm.
    $$
  Applying the definition of $\Gamma_{2,K}^\star[f]$ we get
  \begin{align*}
    (1-\eps)\int_X \varphi_n^2\,\d\Gamma_{2,K}^\star[f]\le 
    \frac 1{4^n\eps}\cE(f)+2^{-n}\|\Delta f\|_{L^2} \sqrt{\cE(f)}+
    \int_X \Big((\Delta f)^2-K\Gq f\Big)\varphi_n^2\,\d\mm.
  \end{align*}
  Passing to the limit first as $n\to\infty$ and then as $\eps\down0$ 
  we get the bound of statement (a).

  Concerning (b), let us first remark that, setting $g:=\Gq f\in W^{1,2}_{\rm loc}(X,\sfd,\mm)$ and $ g_k:=\min\{g,k\}$,
  by Theorem~\ref{cor:local} we have
  \begin{displaymath}
    \Gq {g_k}\le 4\gamma\, g_k\, \nchi_k,\quad
    \text{where}\quad
    \nchi_k=
    \begin{cases}
      1&\text{in the set }\{g\le k\},\\
      0&\text{in the set }\{g>k\}
    \end{cases}
  \end{displaymath}
  so that $g_k\in W^{1,2}(X,\sfd,\mm)$.
  
  Integrating now the nonnegative function $g_k\varphi_n^2\in W^{1,2}_{\rmc}(X,\sfd,\mm)$ 
  with respect to $\Gamma_{2,K}^\star[f]$ and observing that 
  $\Gbil g{g_k}=\Gq{g_k}\nchi_k=\Gq {g_k}$, from \eqref{eq:12} of Theorem~\ref{thm:crucial} we get
  \begin{align*}
    \int_X \Big(K_\lambda \,g\,g_k+\Gq{g_k}\Big)\varphi_n^2\,\d\mm
    &\le 2\int_X \Big(\Delta_\lambda f\,\Gbil f{g_k}+\Delta f\Delta_\lambda f
    g_k\Big)\varphi_n^2\,\d\mm
    \\&
    +2\int_X \Big(-g_k\Gbil g{\varphi_n}+2\Delta_\lambda
    f\,g_k\Gbil f{\varphi_n}\Big)\varphi_n\,\d\mm,
  \end{align*}
  where $K_\lambda=2K+2\lambda$ and we choose $\lambda$ in such a way that $K_\lambda\geq 1$.
  We can now estimate from above the integrals on the right hand side:
  \begin{align*}
    I&=\int_X \Delta_\lambda f\,\Gbil f {g_k}\varphi_n^2\,\d\mm\le 
    \|\Delta_\lambda f\|_{L^4}\,\|g_k^{1/2}\varphi_n\|_{L^4}\,
    \|\Gq {g_k}^{1/2}\varphi_n\|_{L^2}
    \\&\le 2 \|\Delta_\lambda f\|_{L^4}^4+
    \frac 18 \int_X g^2_k\,\varphi_n^2\,\d\mm+
    \frac 14 \int_X \Gq{g_k}\varphi_n^2\,\d\mm,
    \intertext{since $|\Gbil f{g_k}|\le \sqrt{g\,\Gq {g_k}}\nchi_k=
      \sqrt{g_k\,\Gq {g_k}}\nchi_k$ and $|\varphi_n|\leq 1$;}
    II&=\int_X \Delta f\,\Delta_\lambda f\,g_k\,\varphi_n^2\,\d\mm\le 
    \|\Delta f\|_{L^4}\|\Delta_\lambda f\|_{L^4}\,\|g_k\varphi_n^2\|_{L^2}
    \\&\le 8\|\Delta_\lambda f\|_{L^4}^4+\frac 18 \int_X g_k^2\, \varphi_n^2\,\d\mm
    \qquad\text{by \eqref{eq:39} and $|\varphi_n|\le 1$,}
    \\
    III&=-\int_X g_k\Gbil{g}{\varphi_n}\varphi_n\,\d\mm\le 
    \frac k{2^n} \int_X \sqrt {\Gq g}\,\d\mm
    \le \frac k{2^{n-1}}
    \Big(\Gamma_{2,K}^\star[f](X)\cdot\cE(f)\Big)^{1/2}
    \intertext{where we used \eqref{eq:57} and the finiteness of $\Gamma_{2,K}^\star[f]$,}
    IV&=2\int_X \Delta_\lambda
    f\,g_k\Gbil f{\varphi_n}\varphi_n\,\d\mm
    \le \frac k{2^{n-1}}
    \|\Delta_\lambda f\|_{L^2}\, \cE(f)^{1/2}.
  \end{align*}
  Since $g_k^2\le g\,g_k$, summing the contribution of 
  the four terms and using \eqref{eq:10}, we get
   \begin{align*}
     \frac 12\int_X \Big(K_\lambda \,g\,g_k+\Gq{g_k}\Big)\varphi_n^2\,\d\mm
    &\le 20\|\Delta_\lambda f\|_{L^4}^2 +
    \frac k{2^{n-2}}\cE(f)^{1/2}\Big(\|\Delta_\lambda f\|_{L^2}+\Gamma_{2,K}^\star[f](X)^{1/2}\Big).
  \end{align*}
  Passing first to the limit as $n\to\infty$ we obtain
  \begin{displaymath}
    \int_X \Big(K_\lambda \,g\,g_k+\Gq{g_k}\Big)\,\d\mm
    \le 40\|\Delta_\lambda f\|_{L^4}^2.
  \end{displaymath}
  We then pass to the limit as $k\to\infty$ and we obtain 
  $g\in W^{1,2}(X,\sfd,\mm)$.
  Finally, \eqref{eq:12} can be obtained as in the proof of Theorem~\ref{thm:crucial}.
\end{proof}

\begin{theorem}\label{thm:LTGComp}
Let $(X,\sfd,\mm)\in\XX$, with $(X,\sfd)$ length and locally compact.
Assume that there exists a covering $\{U_i\}_{i\in I}$ of $X$ by non-empty open sets $U_i$
such that $(\bar{U}_i,\sfd,\mm\llcorner{\bar{U}_i})\in \XX$ satisfy 
the metric $\BE(K,N)$ condition. 

Then also $(X,\sfd,\mm)$ satisfies the metric $\BE(K,N)$
condition. 
\end{theorem}
\begin{proof} 
  By applying Lemma~\ref{lem:Approximation}, 
  for every $f\in \dom_\V(\Delta)$ we can find
  $f_k\in \dom_\V(\Delta) \cap \dom_{L^\infty}(\Delta)\subset
  \dom_{L^4}(\Delta)\cap L^\infty(X,\mm)$
  with $f_k\to f$ and $\Delta f_k\to\Delta f$ in $\V$ as $k\to\infty$.
  If $\varphi\in \dom_{L^\infty}(\Delta)$ is nonnegative, an integration by parts and Theorem~\ref{thm:medusa} give
  $$
  \GGamma_2(f_k;\varphi)=
    \int_X\Big( -\frac 12\Gbil{\Gq{f_k}}\varphi+\Delta
    f_k\Gbil{f_k}\varphi+
    (\Delta f_k)^2\varphi\Big)\,\d\mm.
  $$
  Therefore, if $\varphi_n$ is defined by \eqref{eq:65}, we can apply
  Theorem~\ref{cor:local} to get
  \begin{align}\label{eq:september9}
    \GGamma_2(f_k;\varphi)&=\lim_{n\to\infty}
    \int_X\Big( -\frac 12\Gbil{\Gq{f_k}}{\varphi\varphi_n}+\Delta
    f_k\Gbil{f_k}{\varphi\varphi_n}+
    (\Delta f_k)^2{\varphi\varphi_n}\Big)\,\d\mm
    \\&\topref{eq:11}\ge 
    \lim_{n\to\infty}\int_X \Big(K\Gq {f_k}+\nu(\Delta
    f_k)^2\Big)\varphi\varphi_n\,\d\mm
    =\int_X \Big(K\Gq {f_k}+\nu(\Delta
    f_k)^2\Big)\varphi\,\d\mm.\nonumber
  \end{align}
  We can use the convergence
  of $f_k$ and $\Delta f_k$ in $\V$ to obtain that $\GGamma_2(f_k;\varphi)$ converges
  to $\GGamma_2(f;\varphi)$.
  Therefore, passing to the limit as $k\to\infty$ in \eqref{eq:september9} we obtain the
  $\BE(K,N)$ condition.
  
  In order to conclude, it suffices to show
  that any essentially bounded 
  $f\in W^{1,2}(X,\sfd,\mm)$ with
  $\|\weakgrad{f}\|_{L^\infty(X,\mm)}\leq 1$ 
  has a $1$-Lipschitz representative.
  Clearly, the fact that $(\bar{U}_i,\sfd,\mm\llcorner{\bar{U}_i})$ 
  satisfy the metric $\BE(K,\infty)$ condition implies that $f$ has a $1$-Lipschitz
  representative on $\bar{U}_i$, 
  and therefore $f$ has a locally $1$-Lipschitz representative $\tilde
  {f}$ on $\cup_i U_i$. 
  If we consider an
  absolutely continuous curve
  $\gamma$ connecting $x$ to $y$, this easily yields (by a covering argument) 
  $$
  |\tilde{f}(x)-\tilde{f}(y)|\leq {\rm length}(\gamma).
  $$
  Since $(X,\sfd)$ is a length space (in fact geodesic), we conclude.
\end{proof}

\section{$\RCD^*(K,N)$ spaces and their localization and
  globalization}
\label{sec:RCD}

In the next two subsections we introduce the $\RCD(K,\infty)$ and $\RCD^*(K,N)$ spaces, and discuss
their equivalent characterizations as well as their localization and globalization properties; 
the case $N=\infty$ is by now well established \cite{AGS12}, while the dimensional case
is more recent \cite{Erbar-Kuwada-Sturm13}, \cite{Ambrosio-Mondino-Savare13}.

\subsection{The case $N=\infty$}

We say that $(X,\sfd,\mm)\in\XX$ is a $\RCD(K,\infty)$ space if the Shannon entropy 
$\mathcal U_\infty:\ProbabilitiesTwo{X}\to (-\infty,+\infty]$
\begin{equation}\label{eq:shannon}
\mathcal U_\infty(\mu):=
\begin{cases}
\int_X\rho\log\rho\,\d\mm &\text{if $\mu=\rho\mm$;}\\\\
+\infty &\text{otherwise}
\end{cases}
\end{equation}
is convex along Wasserstein geodesics. More precisely, here $\ProbabilitiesTwo{X}$ stands for the space of Borel probability 
measures with finite quadratic moments and condition \eqref{eq:grygorian} guarantees that the negative part of $\rho\log\rho$ is integrable for any
$\mu=\rho\mm\in\ProbabilitiesTwo{X}$, see \cite{AGS11b} for details. Hence, \eqref{eq:shannon} makes sense.

If we endow $\ProbabilitiesTwo{X}$ with the
quadratic Wasserstein distance $W_2$, we say that $(X,\sfd,\mm)\in\XX$ is a $\RCD(K,\infty)$ space if for all  
$\mu_0=\rho_0\mm$, $\mu_1=\rho_1\mm$ in $\ProbabilitiesTwo{X}$ and for every constant speed geodesic $\mu_t$
in $\ProbabilitiesTwo{X}$ from $\mu_0$ to $\mu_1$, for all $t\in [0,1]$ there holds $\mu_t=\rho_t\mm$ and
\begin{equation}\label{eq:Kshannon}
\int_X\rho_t\log\rho_t\,\d\mm\leq (1-t)\int_X\rho_0\log\rho_0\,\d\mm+t\int_X\rho_1\log\rho_1\,\d\mm-\frac{K}{2}t(1-t)W_2^2(\mu_0,\mu_1).
\end{equation}

This class of spaces has been introduced in \cite{AGS11b}, where one of the main results is also the equivalence with another
entropic formulation, based on the so-called ${\sf EVI}_K$ property of the Shannon entropy along the heat flow. The definition adopted here 
has been later on improved in \cite{AGMR12} (asking the convexity inequality along \emph{some} geodesic, and then
recovering convexity along \emph{all} geodesics out of the ${\sf EVI}_K$ theorem \cite{Daneri-Savare08}), see also 
\cite{Rajala-Sturm12} for new recent developments. While this characterization
is extremely useful in the proof of stability properties \cite{AGS11b,AGS12,GMS13}, in the proof of localization or globalization properties it suffers the same limitations 
described in Remark~\ref{rem:BEglobal}. It is instead crucial for us the following connection between $\RCD(K,\infty)$ and
$\BE(K,\infty)$, obtained in \cite{AGS12}.

\begin{theorem}[Equivalence of $\RCD(K,\infty)$ and $\BE(K,\infty)$]\label{thm:BE7}
Let $(X,\sfd,\mm)\in\XX$. Then $(X,\sfd,\mm)$ is $\RCD(K,\infty)$ 
if and only if it satisfies the metric $\BE (K,\infty)$ condition.
\end{theorem}
Notice that the assumption that functions with bounded relaxed gradient have a continuous representative is necessary,
in conjunction with $\BE(K,\infty)$, to have $\RCD(K,\infty)$: this way simple examples where $\C\equiv 0$ and $\BE(K,\infty)$
obviously holds (see for instance \cite[Remark~4.12]{AGS11b}) are ruled out. For the reader's
convenience, we state the Global-to-Local property, see \cite[Theorem~6.20]{AGS11b} for the proof, relying on the fact that
one can find geodesics connecting probability measures in $\bar{U}$ lying entirely in $\bar{U}$.

\begin{proposition} [Global-to-Local for $RCD(K,\infty)$]\label{prop:globaloca} Let $(X,\sfd,\mm)\in\XX$ be $\RCD(K,\infty)$ and let $U\subset X$ be open.
If $\mm(\partial U)=0$ and $(\bar{U},\sfd)$ is geodesic, then $(\bar{U},\sfd,\mm\llcorner\bar{U})$ is $\RCD(K,\infty)$.
\end{proposition}

The proof of the Local-to-Global property, established under the non-branching condition in \cite{Sturm06I}, heavily relies
on the $\BE(K,\infty)$ characterization of Theorem~\ref{thm:BE7}. 
Notice that the only global assumptions are
\eqref{eq:grygorian} and the length property (necessary already
for subsets of Euclidean spaces).

\begin{theorem} [Local-to-Global for $\RCD(K,\infty)$] \label{thm:locagloba2}
  Let $(X,\sfd,\mm)\in\XX$ be a length and locally compact 
space and assume that there exists a covering $\{U_i\}_{i\in I}$  of
$X$ by non-empty open 
subsets such that
$(\bar{U}_i,\sfd,\mm\llcorner{\bar{U}_i})\in \XX$ satisfy $\RCD(K,\infty)$. 

Then $(X,\sfd,\mm)$ is a $\RCD(K,\infty)$ space. 
\end{theorem}
The \emph{proof} is an immediate consequence of Theorem~\ref{thm:BE7}
and Theorem~\ref{thm:LTGComp}. 

\subsection{The case $N<\infty$}

For $N\geq 1$, let $U_N:[0,\infty)\to\R$ be defined by $U_N(r):=N(r-r^{1-1/N})$.
The induced dimension-dependent R\'enyi entropy functionals $\mathcal U_N$ (whose limit as $N\to\infty$ 
is the Shannon entropy $\mathcal U_\infty$ in \eqref{eq:shannon}) is defined by 
\begin{equation}
  \label{eq:41}
  \mathcal U_N(\mu):=\int_X U_N(\varrho)\,\d\mm+N\mu^\perp(X)
  \quad\text{if }\mu=\varrho\mm+\mu^\perp, \quad
  \mu^\perp\perp\mm.
\end{equation}
Since $U_N(0)=0$ and the negative part of $U$ grows
at most linearly, $\mathcal U_N$ is well defined and with values in $\R$ if $\mu$ has bounded support.

We now introduce, for $\kappa\in\R$, the distortion coefficients
\begin{equation}
  \label{eq:170}
  \sigma^{(t)}_\kappa(\delta):=
  \begin{cases}
    +\infty&\text{if }\kappa\ge \pi^2,\\
    \displaystyle
    \frac{\sin(t\sqrt \kappa \delta)}{\sin(\sqrt\kappa\delta)}&
    \text{if }0<\kappa<\pi^2\\
    t&\text{if }\kappa=0,\\
    \displaystyle
    \frac{\sinh(t\sqrt{- \kappa} \delta)}{\sin(\sqrt{-\kappa}\delta)}&
    \text{if }\kappa<0.
  \end{cases}
\end{equation}

The so-called $\CD^*(K,N)$ condition introduced by Bacher and Sturm in \cite{Bacher-Sturm10} is based, in analogy with the case $N=\infty$, on
a convexity inequality of $\mathcal U_N$ along Wasserstein geodesics; it is a variant of the $\CD(K,N)$ condition originally introduced 
by Sturm and studied in \cite{Lott-Villani09}, \cite{Sturm06I,Sturm06II} (based on a different choice of the distortion coefficients in \eqref{eq:177} below). Here we just mention that $\CD_{{\rm loc}}(K,N)$ is equivalent to $\CD^*(K,N)$, and this
fact strongly suggests that the latter should have better globalization/localization properties. For the purpose of this paper,
we just define the ``Riemannian'' $\CD^*(K,N)$ condition,  adding the condition that $\C$ is a quadratic form.

\begin{definition}[$\RCD^*(K,N)$ condition]
  For $K\in\R$ and $N\in [1,\infty)$, we say that $(X,\sfd,\mm)\in\XX$ satisfies the
  $\RCD^*(K,N)$ condition if 
  for every $\mu_i=\varrho_i\mm\in \Probabilities X$, $i=0,1$,
  with bounded support, for all constant speed geodesic $\mu_s:[0,1]\to\ProbabilitiesTwo{X}$ from $\mu_0$ to
  $\mu_1$ and 
  for every $M\ge N$ there holds 
  \begin{equation}
    \label{eq:177}
    \mathcal U_M(\mu_s)\le 
    \int \Big(\sigma_{K/M}^{(1-s)}(\sfd(\gamma_0,\gamma_1))
    \varrho_0(\gamma_0)^{-1/M}+
    \sigma_{K/M}^{(s)}(\sfd(\gamma_0,\gamma_1))
    \varrho_1(\gamma_1)^{-1/M}\Big)\,\d\ppi(\gamma),
  \end{equation}
  where $\sigma_\kappa$ is defined in \eqref{eq:170} and $\mathcal U_M$ is defined in \eqref{eq:41}.  
 \end{definition}

The following result, extending \cite{AGS12} to the dimensional case, has been proved in \cite{Ambrosio-Mondino-Savare13}
using, from this paper, only the ``abstract'' regularity estimates in $\BE(K,N)$ Dirichlet spaces derived in \S3; 
see also Remark~\ref{rem:otherRCD} below for the closely related result \cite{Erbar-Kuwada-Sturm13}.

\begin{theorem}[Equivalence of $\RCD^*(K,N)$ and $\BE(K,N)$]\label{thm:BE8}
If $(X,\sfd,\mm)\in\XX$ satisfies the metric $\BE(K,N)$ condition 
then $(X,\sfd,\mm)$ is $\RCD^*(K,N)$. The converse holds if $\mm(X)$
is finite or $K\ge 0$.
\end{theorem}

\begin{remark}[Other characterizations of $\RCD^*(K,N)$]\label{rem:otherRCD}
\upshape
{\rm As in the case $N=\infty$, another characterization of $\RCD^*(K,N)$ has been given in
 \cite{Ambrosio-Mondino-Savare13} in terms of suitable Evolution Variation Inequalities (${\sf EVI}$)
satisfied by the gradient flow of the Reny entropy $\mathcal U_N$, with a modulus of continuity proportional 
to $K$ and dependent on $N$ (in the limit case $N=\infty$ the modulus is proportional to the squared Wasserstein distance). 
This requires a detailed analysis of the gradient flow of the Reny entropy $\mathcal U_N$, a nonlinear diffusion equation.
In this connection, 
a remarkable result obtained in \cite{Erbar-Kuwada-Sturm13} is the characterization of the $\BE(K,N)$ property in terms of an ${\sf EVI}$ property 
fulfilled, along the heat flow, by the modified Shannon entropy
$$
\tilde{\mathcal U}_N(\mu):={\rm exp}\biggl(-\frac 1 N\mathcal U_\infty(\mu)\biggr).
$$
This has the advantage of avoiding many technical difficulties related to nonlinear diffusion equations in metric measure spaces.
The definition of $\RCD^*(K,N)$ adopted in \cite{Erbar-Kuwada-Sturm13} is actually based on this ${\sf EVI}$ property and,
due to the equivalence with $\BE(K,N)$ proved in that paper, our Local-to-Global result applies to this definition as well.

However, as we discussed in the previous subsection, all these ${\sf EVI}_{K,N}$ formulations, while technically important to get stability and
convexity properties on \emph{all} geodesics, are less relevant in the study of localization/globalization
properties.
}\fr
\end{remark}

\begin{proposition} [Global-to-Local for $\RCD^*(K,N)$] Let $(X,\sfd,\mm)\in\XX$ be $\RCD^*(K,N)$ and let $U\subset X$ be open.
If $\mm(\partial U)=0$ and $(\bar{U},\sfd)$ is geodesic, then $(\bar{U},\sfd,\mm\llcorner\bar{U})$ is $\RCD^*(K,N)$.
\end{proposition}
\begin{proof}
The proof follows the same lines of Proposition~\ref{prop:globaloca}: first (independently of curvature assumptions) we obtain
from Proposition~\ref{prop:localK}(b) that the condition $(X,\sfd,\mm)$ localizes to $U$. Then, we use the fact that
one can find geodesics connecting probability measures in $\bar{U}$ lying entirely in $\bar{U}$.
\end{proof}

\begin{theorem} [Local-to-Global for $\RCD^*(K,N)$] 
Let $(X,\sfd,\mm)\in\XX$  be a length 
space and assume that there exists a covering $\{U_i\}_{i\in I}$ of
$X$ by 
non-empty open subsets such that $\mm(U_i)<\infty$ if $K<0$, 
and $(\bar{U}_i,\sfd,\mm\llcorner{\bar{U}_i})\in \XX$ satisfy $\RCD^*(K,N)$. 

Then $(X,\sfd,\mm)$ is a $\RCD^*(K,N)$ space. 
\end{theorem}
\begin{proof} Since $\mm(\bar{U}_i)<\infty$ if $K<0$, we know from Theorem~\ref{thm:BE8}
  that all spaces $(\bar{U}_i,\sfd,\mm\llcorner{\bar{U}_i})$ are
  metrically $\BE(K,N)$ and they are also locally compact,
  so that $(X,\sfd)$ is locally compact. 
  Therefore 
  Theorem~\ref{thm:LTGComp} applies and shows that $(X,\sfd,\mm)$
  is a metrically
  $\BE(K,N)$ space and we conclude 
  applying Theorem~\ref{thm:BE8} once more.
\end{proof}

\def\cprime{$'$}

\end{document}